\documentclass[11pt,a4paper,english]{article}

\usepackage{amsfonts}
\usepackage{amsthm}
\usepackage{amsmath}
\usepackage{amssymb}
\usepackage{eufrak}
\usepackage{graphicx}
\usepackage{subfigure}
\usepackage{type1cm}
\usepackage{color}
\usepackage[sort]{cite}
\usepackage[a4paper, total={6in, 8in}]{geometry}

\newtheorem{theorem}{Theorem}

\newtheorem{lemma}[theorem]{Lemma}

\newtheorem{definition}[theorem]{Definition}

\newtheorem{remark}[theorem]{Remark}

\makeatletter
  \renewenvironment{proof}[1][\proofname]{\par
    \pushQED{}%
    \normalfont \topsep6\p@\@plus6\p@\relax
    \trivlist
    \item\relax
          {\itshape
      #1\@addpunct{.}}\hspace\labelsep\ignorespaces
  }{%
    \popQED\endtrivlist\@endpefalse
  }
  \makeatother

\title{Asymptotic expansions of high-frequency multiple scattering iterations for
sound hard scattering problems\thanks{Y. Boubendir’s work was supported by
the NSF through Grants DMS-1720014 and] DMS-2011843. F. Ecevit is supported
by the Scientific and Technological Research Council of Turkey through grant
T\"{U}B\.{I}TAK-1001-117F056.}}

\author{Yassine Boubendir\thanks{New Jersey Institute of Technology,
Department of Mathematical Sciences, University Heights, Newark NJ 07102,
USA. (boubendi@njit.edu)} \and
Fatih Ecevit\thanks{ Bo\u{g}azi\c{c}i University, Department of Mathematics,
34342 Bebek, Istanbul, Turkey. (fatih.ecevit@boun.edu.tr)}}

\begin{document}

\maketitle

\begin{abstract}
We consider the two-dimensional high-frequency plane wave scattering problem in the exterior of a finite collection of disjoint, compact, smooth, strictly convex obstacles with Neumann boundary conditions. Using integral equation
formulations, we determine the H\"{o}rmander classes and derive high-frequency asymptotic expansions of the total fields corresponding to multiple scattering iterations on the boundaries of the scattering obstacles. These asymptotic expansions are used to obtain sharp wavenumber dependent estimates on the derivatives of multiple scattering total fields which, in turn, allow for the optimal design and numerical analysis of Galerkin boundary element methods for the efficient (frequency independent) approximation of sound hard multiple scattering returns. Numerical experiments supporting the validity of these expansions are presented.
\end{abstract}



\section{Introduction}
\label{intro}

High-frequency scattering problems considered in this
manuscript have been effectively tackled using asymptotic approaches such as
the ray method (RM) \cite{Lewis65} and geometrical optics (GO) \cite{KlineKay65}.
These methods are based on ray ans\"{a}tze in the short-wavelength limit which
take the form of asymptotic series of amplitudes in inverse powers of the
wavenumber $k$ modulated by the oscillations in the incident field of radiation.
The resulting eikonal equation for the phase and the recursive system of transport
equations for the amplitudes are independent of frequency. The diffraction effects
ignored in RM and GO are taken into account in refined approaches such as the
uniform theory of diffraction \cite{McNamaraEtAl90} and geometrical theory
of diffraction \cite{Keller62,BabicBuldyrev91,BorovikovKinber94,Keller95,Tew00}
(see \cite{Runborg07} for a general review). While implementations based upon
asymptotic methods are frequency independent and the accuracy increases with
increasing frequency, they are not designed to converge for fixed frequencies.
On the other hand, in the case of low to moderate frequencies,
classical schemes based on direct discretizations of differential equations
such as the finite difference time domain method \cite{JovanovicSuli14,LeVeque07},
variational formulations including the method of moments \cite{Gibson15}, the finite
element method \cite{ErnGuermond21,Jin02,SunZhou17}, the finite volume
method \cite{LeVeque02}, and integral equation formulations
\cite{Nedelec01,ColtonKress13} including the accelerated ones using hierarchical
matrices \cite{BanjaiHackbusch05,Borm10,Seibel22}, the fast multipole method
\cite{Liu09} have been successfully used in the design of numerical algorithms
for wave propagation problems.

Methods that combine the advantages of classical schemes (error controllability)
and asymptotic methods (frequency independent degrees of freedom) and that thereby
provide simulation strategies applicable over the entire frequency spectrum
have been the content of increasingly active research in the last decades (see e.g.
\cite{ChungEtAl17,LamEtAl19,Nguyen15} and the references therein).
In this context, approaches that combine asymptotic expansions and integral
equation formulations have displayed the capability of delivering frequency
independent accuracies with the utilization of numbers of degrees of freedom
that needs to increase only mildly with increasing frequency. Moreover some
of them are frequency independent. These methods can be classified
as corresponding to single or multiple scattering problems.
For the case of single scattering, the problems considered correspond
to smooth strictly convex obstacles \cite{DominguezEtAl07,EcevitOzen17,Ecevit18,EcevitEruslu19,EcevitEtAl22},
convex polygons \cite{ArdenEtAl07,Chandler-WildeLangdon07,Chandler-WildeEtAl07,HewettEtAl13,GrothEtAl15,GrothEtAl18}, screens and apertures \cite{HewettEtAl15,GibbsEtAl20},
and half-planes \cite{LangdonChandler-Wilde06,Chandler-WildeEtAl04},
see also \cite{Chandler-WildeGraham09,Chandler-WildeEtAl12,Chandler-WildeLangdon15}
for general reviews and \cite{Galkowski19,GalkowskiEtAl19,SpenceEtAl15} for high-frequency properties
of integral operators. 
On the other hand, algorithms relating to
multiple scattering configurations are significantly more demanding
due to the nature of the problem. In this
connection, the boundary element methods proposed for a convex polygon
in addition to several small obstacles \cite{GibbsEtAl21} and a
non-convex polygon \cite{Chandler-WildeEtAl15} both demand an
$\mathcal{O}(\log k)$ increase in the number of DoF to maintain accuracy
with increasing frequency. While, on the one hand, the method in \cite{GibbsEtAl21} does
not require an iteration (e.g. through a utilization of a Neumann series)
and thus ray tracing, it is not applicable over the entire frequency spectrum
since small obstacles are required to have sizes on the order of the wavelength
to preserve the $\mathcal{O}(\log k)$ efficiency. The case of a non-convex polygon
considered in \cite{Chandler-WildeEtAl15} allows for only
finitely many reflections and thus a finite ray tracing procedure.

The case of several smooth compact strictly convex scatterers (with no restriction
on their sizes with respect to the wavelength) treated in this paper is a
multiple scattering problem resulting in infinitely many reflections and trapping
relations. The related scattering problem was initially studied in the context
of the Dirichlet boundary condition \cite{BrunoEtAl05,EcevitReitich09,AnandEtAl10,BoubendirEtAl17}.
In these papers the multiple scattering formulation is based on a Neumann series
decomposition applied to integral equation formulations (in two-
\cite{BrunoEtAl05,EcevitReitich09,BoubendirEtAl17} and three-dimensions \cite{AnandEtAl10}).
Here, in the case of the Neumann (sound hard) boundary condition, we reduce the multiple
scattering problem to a collection of single scattering partial differential
equations. Our main contributions are the derivation of the asymptotic expansions
of the total fields associated with multiple scattering iterations on the boundaries
of scattering obstacles, and of sharp wavenumber dependent estimates on the
derivatives of these densities. As we shall explain, these estimates are fundamental
in the design and numerical analysis of efficient (frequency independent) boundary
element methods for the iterated solutions of multiple scattering problems.
In addition, we present numerical tests validating the asymptotic expansions derived.

As we mentioned, the two-dimensional multiple scattering problem we consider
here was studied in \cite{EcevitReitich09} for the Dirichlet boundary
condition. The asymptotic expansions derived therein regarding
the normal derivative of the multiple scattering total fields were used in
\cite{EcevitEtAl20} to design efficient (frequency independent) Galerkin boundary
element methods for the Dirichlet multiple scattering problem.
Generally speaking, the approach in \cite{EcevitEtAl20} was an extension of the
single scattering algorithms \cite{EcevitOzen17,EcevitEruslu19} to multiple scattering
problems.

For a plane wave incidence $u^{\rm inc}(x,k) = e^{ik \, \alpha \cdot x}$ with direction
$\alpha$ impinging on a single smooth convex obstacle subject to the Dirichlet boundary
condition, the key elements in the design of Galerkin approximation spaces 
in \cite{EcevitOzen17,EcevitEruslu19} were phase extraction
\[
	\mu(x,k)
	= e^{ik \, \alpha \cdot x} \mu^{\rm slow}(x,k)
\]
(where $\mu$ is the unknown normal derivative of the total field) and the
Melrose-Taylor \cite{MelroseTaylor85,EcevitReitich09} asymptotic expansion
\begin{equation}
	\label{eq:MTDslow}
	\mu^{\rm slow}(x,k)
	\sim \sum_{p,q \ge 0} k^{\frac{2}{3}-\frac{2p}{3}-q} \, b_{p,q}(x) \, \Psi^{(p)}(k^{\frac{1}{3}} Z(x))
\end{equation}
(see \cite{LazerguiBoubendir17} for an alternative approach) used to derive
sharp wavenumber explicit estimates on the derivatives of the envelope
$\mu^{\rm slow}$. 
The resulting Galerkin boundary element methods
were shown to demand an increase of only $\mathcal{O}(k^{\epsilon})$, for any $\epsilon>0$,
in the number of DoF to maintain accuracy with increasing frequency. Moreover,
as shown in \cite{Ecevit18}, the methods in \cite{EcevitOzen17,EcevitEruslu19}
are frequency independent, i.e. $\mathcal{O}(1)$, provided a sufficient number of terms
in the asymptotic expansion is incorporated into integral equation formulations.

In the case of the Neumann boundary condition, we have recently developed
Galerkin boundary element methods \cite{EcevitEtAl20} for a single
smooth convex obstacle. Similar to the Dirichlet case \cite{EcevitOzen17,Ecevit18,EcevitEruslu19},
the approach is based on phase extraction
\[
	\eta(x,k)
	= e^{ik \, \alpha \cdot x} \eta^{\rm slow}(x,k)
\]
(where $\eta$ is the unknown total field) and the utilization of
Melrose-Taylor \cite{MelroseTaylor85,EcevitEtAl22} asymptotic expansion
\begin{equation}
	\label{eq:MTNslow}
	\eta^{\rm slow}(x,k)
	\sim
	\sum_{p,q,r \ge 0 \atop \ell \le -1} k^{-\frac{1+2p+3q+r+\ell}{3}+(\ell+1)_-} \, b_{p,q,r,\ell}(x) \,
	(\Psi^{r,\ell})^{(p)}(k^{\frac{1}{3}}Z(x))
\end{equation}
for the derivation of sharp wavenumber explicit estimates on the derivatives of
$\eta^{\rm slow}$.
To the best of our knowledge, the single scattering problem for the Neumann boundary
condition was explored only recently \cite{EcevitEtAl22} due to the complicated form of the asymptotic expansion \eqref{eq:MTNslow} when compared to its Dirichlet counterpart \eqref{eq:MTDslow}. For this reason, as we will see, the extension of the asymptotic expansion \eqref{eq:MTNslow} to multiple scattering problems presents more difficulties than
the extension of the Dirichlet expansion  considered in \cite{EcevitReitich09}. 

The asymptotic expansions we develop
here in this paper form the main element in the extension of the Galerkin boundary
element methods proposed in \cite{EcevitEtAl20} for the Neumann boundary condition
to the case of several smooth convex obstacles for the efficient (frequency independent)
approximation of multiple scattering iterations. Another application of these expansions
is the possibility of studying of the rate of convergence of multiple scattering iterations as done
in \cite{EcevitReitich09} for the Dirichlet case.

The paper is organized as follows. In Sect. \ref{sect:2}, we introduce the
sound hard scattering problem, discuss the single scattering Melrose-Taylor
asymptotic expansion of the total field, and present the multiple scattering
formulation. In Sect. \ref{sect:3}, we set the technical assumptions and
summarize the geometry of multiple scattering rays. In the same section,
we present two of the main results of the paper. They concern the asymptotic expansions
of the multiple scattering total fields on the boundaries of the scattering obstacles
(see Theorem \ref{Thm:main1}) and sharp wavenumber dependent estimates on
their derivatives (cf. Theorem \ref{thm:etamder}). Sect. \ref{sect:4} is reserved
for the statement and proof of the third main result relating to the asymptotic
expansions of the iterated scattered fields (see Theorem \ref{Thm:main2}).
Finally, numerical results validating the asymptotic expansions derived in 
Theorem \ref{Thm:main1} are presented in Sect. \ref{sect:5}.

\section{Sound hard multiple scattering problems}
\label{sect:2}

We consider the two dimensional sound hard multiple scattering problem in the exterior of the disjoint union
$K = \bigcup \{ K_\sigma : \sigma \in \mathcal{J} \}$ of finitely many smooth compact strictly convex obstacles
illuminated by a plane wave incidence $u^{\rm inc}(x,k) = e^{ik \, \alpha \cdot x}$ ($k>0$ and $|\alpha| = 1$).
The unknown scattered field $u$ satisfies the sound hard scattering problem \cite{ColtonKress92,Chandler-WildeEtAl12}
\begin{equation} \label{eq:Helmholtz}
	\left\{
		\begin{array}{l}
			(\Delta +k^2) u = 0
			\text{ in } \mathbb{R}^2 \backslash K,
			\\
			\partial_{\nu} u = -\partial_{\nu} u^{\rm inc}
			\text{ on } \partial K,
			\\
			\lim\limits_{r \to \infty} \sqrt{r} \big( \frac{\partial u}{\partial r} -iku \big) =0,
			\ r=|x|,
		\end{array}
	\right.
\end{equation}
where $\nu$ is the exterior unit normal to $\partial K$. The direct approach in integral equation formulations of the scattering
problem \eqref{eq:Helmholtz} is based on utilization of the Neumann-to-Dirichlet map through the double-layer representation
\begin{equation} \label{eq:usdl}
	u(x) = \int_{\partial K} \dfrac{\partial G(x,y)}{\partial \nu(y)} \, \eta(y) \, d s(y),
	\qquad
	x \in \mathbb{R}^2 \setminus K,
\end{equation}
where
\begin{equation} \label{eq:fundsol}
	G(x,y) = \dfrac{i}{4} \ H_{0}^{(1)}(k|x-y|)
\end{equation}
is the outgoing Green function for the Helmholtz equation.
Specifically, \eqref{eq:usdl} transforms the Neumann boundary value problem
\eqref{eq:Helmholtz} into the determination of Dirichlet data, namely the total field $\eta = u+ u^{\rm inc}$, on the boundary
$\partial K$.


In this work, we are interested in the high-frequency solution of the scattering problem \eqref{eq:Helmholtz}.
Indeed, for the case of a single smooth compact and strictly convex obstacle,
we have recently shown \cite{EcevitEtAl22} that frequency independent solutions
can be attained resorting to boundary integral equation reformulations of the sound hard
scattering problem. The design of Galerkin boundary element methods we
developed in \cite{EcevitEtAl22} are based on a detailed study of the H\"{o}rmander
class and asymptotic expansion of the total field on the boundary as $k \to \infty$. In detail,
referring to \cite{EcevitReitich09} and
the references therein for the definition of H\"{o}rmander classes $S^\mu_{\varrho,\delta}$, asymptotic expansions,
and rapid decrease in the sense of Schwartz, the asymptotic behavior of the envelope
\[
	\eta^{\rm slow}(x,k) = e^{-ik \, \alpha \cdot x} \, \eta(x,k)
\]
reads as follows; see \cite[Theorem 1]{EcevitEtAl22} and \cite[Theorem 9.36]{MelroseTaylor85}.

\begin{theorem}  \label{Thm:MT85}
If $K \subset \mathbb{R}^{2}$ is a single smooth, compact and strictly convex obstacle, then:
\begin{itemize}
\item[\emph{(i)}] On the illuminated region $\partial K^{\rm IL} = \{ x \in \partial K : \alpha \cdot \nu(x) < 0 \}$,
$\eta^{\rm slow}(x,k) \in S^{0}_{1,0}(\partial K^{\rm IL} \times (0,\infty))$ and has an asymptotic expansion
\begin{equation}
\label{eq:MT85aa}
	\eta^{\rm slow}(x,k)
	\sim \sum_{j \ge 0} k^{-j} a_{j}(x) 
\end{equation}
for some complex-valued smooth functions $a_{j}$.
\item[\emph{(ii)}] On a $(k$ independent$)$ neighborhood $\partial K^{\rm SB}_\epsilon$ of the shadow boundary
$\partial K^{\rm SB} = \{ x \in \partial K : \alpha \cdot \nu(x) = 0 \}$,
$\eta^{\rm slow}(x,k) \in S^0_{\frac{2}{3},\frac{1}{3}}(\partial K^{\rm SB}_\epsilon,(0,\infty))$, and has
an asymptotic expansion 
\begin{align}
	\eta^{\rm slow}(x,k)
	& \sim
	\sum_{p,q,r \ge 0 \atop \ell \le -1} a_{p,q,r,\ell}(x,k)
	\label{eq:MT85}
	\\
	& = \sum_{p,q,r \ge 0 \atop \ell \le -1} k^{-\frac{1+2p+3q+r+\ell}{3}+(\ell+1)_-} \,
	b_{p,q,r,\ell}(x) \, (\Psi^{r,\ell})^{(p)}(k^{\frac{1}{3}}Z(x))
	\nonumber
\end{align}
where $t_- = \min \{ t,0 \}$, $b_{p,q,r,\ell}$ and $\Psi^{r,\ell}$ are complex-valued smooth functions, 
$Z$ is a real-valued smooth function positive on the illuminated region, negative on the shadow region
$\partial K^{\rm SR} = \{ x \in \partial K : \alpha \cdot \nu(x) > 0 \}$,
and vanishes exactly to first order on the  shadow boundary, and the functions $\Psi^{r,\ell}$ are complex-valued smooth functions
with an asymptotic expansion~\cite[Lemma 9.9]{MelroseTaylor85}
\begin{equation*}
	\Psi^{r,\ell}(\tau) \sim \sum_{j=0}^{\infty} a_{r,\ell,j} \tau^{1+\ell-2r-3j}
	\qquad
	\text{as } \tau \to \infty,
\end{equation*}
and they rapidly decrease in the sense of Schwartz as $\tau \to -\infty$.
\item[\emph{(iii)}] On any compact subset of the shadow region, $\eta^{\rm slow}$ rapidly decreases in the sense of Schwartz
as $k \to \infty$.
\item[\emph{(iv)}] Moreover, $\eta^{\rm slow}(x,k) \in S^{0}_{\frac{2}{3},\frac{1}{3}}(\partial K \times (0,\infty))$ and the
asymptotic expansion \eqref{eq:MT85} is valid over the entire boundary $\partial K$.
\end{itemize}
\end{theorem}

The main goal of this paper is the extension of Theorem~\ref{Thm:MT85} to multiple scattering problems.
To this end, we first observe that, when $|\mathcal{J}| \ge 2$, the scattered field $u$ can be uniquely
decomposed into
\begin{equation} \label{eq:decomposeu}
	u = \sum_{\sigma \in \mathcal{J}} u_\sigma
	\quad
	\text{in }
	\mathbb{R} \setminus K 
\end{equation}
where $\{ u_\sigma : \sigma \in \mathcal{J} \}$ solve the coupled system of sound hard scattering problems
\begin{equation}
\label{eq:Helmholtzsigma}
	\left\{
		\begin{array}{l}
			(\Delta +k^2) u_\sigma = 0
			\text{ in } \mathbb{R}^2 \setminus K_\sigma,
			\\
			\partial_{\nu} u_\sigma = - \partial_{\nu} (u^{\rm inc} + \sum_{\tau \in \mathcal{J} \setminus \{ \sigma \}} u_\tau)
			\text{ on } \partial K_\sigma,
			\\
			\lim\limits_{r \to \infty} \sqrt{r} \big( \frac{\partial u_\sigma}{\partial r} -iku_\sigma \big) = 0.
		\end{array}
	\right.
\end{equation}
Moreover, we formally have
\[
	u_\sigma = \sum_{m=0}^\infty u_\sigma^m
\]
where
\begin{equation}
\label{eq:Helmholtzsigmadecomp0}
	\left\{
		\begin{array}{l}
			(\Delta +k^2) u_\sigma^0 = 0
			\text{ in } \mathbb{R}^2 \setminus K_\sigma,
			\\
			\partial_{\nu} u_\sigma^0 = - \partial_{\nu} u^{\rm inc}
			\text{ on } \partial K_\sigma,
			\\
			\lim\limits_{r \to \infty} \sqrt{r} \big( \frac{\partial u_\sigma^0}{\partial r} -iku_\sigma^0 \big) =0,
		\end{array}
	\right.
\end{equation}
and for $m \ge 1$
\begin{equation}
\label{eq:Helmholtzsigmadecompm}
	\left\{
		\begin{array}{l}
			(\Delta +k^2) u_\sigma^m = 0
			\text{ in } \mathbb{R}^2 \setminus K_\sigma,
			\\
			\partial_{\nu} u_\sigma^m = - \sum_{\tau \in \mathcal{J} \setminus \{ \sigma \}} \partial_\nu u_\tau^{m-1}
			\text{ on } \partial K_\sigma,
			\\
			\lim\limits_{r \to \infty} \sqrt{r} \big( \frac{\partial u_\sigma^m}{\partial r} -iku_\sigma^m \big) =0.
		\end{array}
	\right.
\end{equation}
It follows from \eqref{eq:Helmholtzsigmadecomp0}-\eqref{eq:Helmholtzsigmadecompm} that, in fact, we have
\begin{equation} \label{eq:decomposeusigmam}
	u_\sigma^m
	= \sum_{\tau_0,\ldots,\tau_{m} \in \mathcal{J} \atop \tau_j \ne \tau_{j-1}, \, \tau_m = \sigma}
	u_{\tau_0,\ldots,\tau_m}
\end{equation}
where $u_{\tau_0},\ldots,u_{\tau_0,\ldots,\tau_m}$ iteratively solve the sound hard scattering problems
\begin{equation}
\label{eq:Helmholtzsigmadecomp0new}
	\left\{
		\begin{array}{l}
			(\Delta +k^2) u_{\tau_0} = 0
			\text{ in } \mathbb{R}^2 \setminus K_{\tau_0},
			\\
			\partial_{\nu} u_{\tau_0} = - \partial_{\nu} u^{\rm inc}
			\text{ on } \partial K_{\tau_0},
			\\
			\lim\limits_{r \to \infty} \sqrt{r} \big( \frac{\partial u_{\tau_0}}{\partial r} -iku_{\tau_0} \big) =0,
		\end{array}
	\right.
\end{equation}
and for $j = 1,\ldots,m$
\begin{equation}
\label{eq:Helmholtzsigmadecompmnew}
	\left\{
		\begin{array}{l}
			(\Delta +k^2) u_{\tau_0,\ldots,\tau_j} = 0
			\text{ in } \mathbb{R}^2 \setminus K_{\tau_j},
			\\
			\partial_{\nu} u_{\tau_0,\ldots,\tau_j} = - \partial_\nu u_{\tau_0,\ldots,\tau_{j-1}}
			\text{ on } \partial K_{\tau_j},
			\\
			\lim\limits_{r \to \infty} \sqrt{r} \big( \frac{\partial u_{\tau_0,\ldots,\tau_j}}{\partial r} -iku_{\tau_0,\ldots,\tau_j} \big) =0.
		\end{array}
	\right.
\end{equation}
Using \eqref{eq:decomposeusigmam} in \eqref{eq:decomposeu}, we therefore see that it is sufficient to consider an arbitrary
sequence $\{ K_m \}_{m \ge 0}$ of obstacles in $\{ K_\sigma : \sigma \in \mathcal{J} \}$ satisfying $K_{m+1} \ne K_m$ and
discuss the iterative solution of the sound hard scattering problems
\begin{equation}
\label{eq:Helmholtzsigmadecomp0newer}
	\left\{
		\begin{array}{l}
			(\Delta +k^2) u_0 = 0
			\text{ in } \mathbb{R}^2 \setminus K_0,
			\\
			\partial_{\nu} u_0 = - \partial_{\nu} u^{\rm inc}
			\text{ on } \partial K_0,
			\\
			\lim\limits_{r \to \infty} \sqrt{r} \big( \frac{\partial u_0}{\partial r} -iku_0 \big) =0,
		\end{array}
	\right.
\end{equation}
and for $m \ge 1$
\begin{equation}
\label{eq:Helmholtzsigmadecompmnewer}
	\left\{
		\begin{array}{l}
			(\Delta +k^2) u_m = 0
			\text{ in } \mathbb{R}^2 \setminus K_m,
			\\
			\partial_{\nu} u_m = - \partial_\nu u_{m-1}
			\text{ on } \partial K_m,
			\\
			\lim\limits_{r \to \infty} \sqrt{r} \big( \frac{\partial u_m}{\partial r} -iku_m \big) =0.
		\end{array}
	\right.
\end{equation}
In light of \eqref{eq:usdl}, we see that the solutions $u_m$ of the sound hard multiple scattering problems
\eqref{eq:Helmholtzsigmadecomp0newer}-\eqref{eq:Helmholtzsigmadecompmnewer} can be recovered through
the double layer representations
\begin{equation} \label{eq:usjdl}
	u_m(x,k) = \int_{\partial K_m} \dfrac{\partial G(x,y)}{\partial \nu(y)} \, \eta_m(y,k) \, d s(y),
	\qquad
	x \in \mathbb{R}^2 \setminus K_m,
	\qquad
	m \ge 0,
\end{equation}
provided the \emph{multiple scattering total fields}
\begin{equation} \label{eq:multscattotalfields}
	\eta_m = 
	\left\{
		\begin{array}{ll}
			u_0 + u^{\rm inc},
			& m = 0,
			\\
			u_m+u_{m-1},
			& m \ge 1,
		\end{array}
	\right.
	\qquad
	\text{on } \partial K_m,
\end{equation}
are predetermined. As in the case of single scattering problems, we observe that these predeterminations can be based on frequency independent implementations (utilizing
e.g. Galerkin boundary element methods \cite{EcevitEtAl22}) provided the asymptotic
expansions of the multiple scattering total fields $\eta_m$ are properly incorporated into
the solution strategy. In the next section, under appropriate assumptions, we derive these
asymptotic expansions. As is apparent from the integral equation \cite{ColtonKress13}
\begin{equation} \label{eq:inteqall}
	\eta_m(x,k) - 2 \int_{\partial K_m} \dfrac{\partial G(x,y)}{\partial \nu(y)} \, \eta_m(y,k) \, d s(y)
	= 2
	\left\{
		\begin{array}{ll}
			u^{\rm inc}(x,k),
			& \ m = 0,
			\\
			u_{m-1}(x,k),
			& \ m \ge 1,	
		\end{array}
	\right.
	\quad
	x \in \partial K_m,
\end{equation}
these derivations demand a through understanding of the asymptotic behavior of the scattered fields $u_m$.

\section{Geometry of multiple scattering rays, phase extraction, and asymptotic expansions}
\label{sect:3}
 
For the developments that follow, we assume that

\begin{itemize}
\item[\bf A] The sequence $\{ K_m \}_{m \ge 0} \subset \{ K_\sigma : \sigma \in \mathcal{J} \}$ satisfies $K_{m+1} \ne K_m$
($m \ge 0$) along with the \emph{no-occlusion condition}
\begin{equation}
\label{eq:NoOcclusion}
	\{ x + t \alpha : x \in K_0, t \ge 0 \} \cap K_{1} = \varnothing,
\end{equation}
and the \emph{visibility condition}
\begin{equation}
\label{eq:VisibleVisible}
	K_{m+1} \cap \overline{co}(K_m \cup K_{m+2}) 
	= \varnothing,
	\qquad
	m \ge 0,
\end{equation}
where $\overline{co}$ denotes the closed convex hull.
\item[\bf B] Theorem~\ref{Thm:MT85} holds for incident fields $w$ impinging on
smooth compact strictly convex obstacles $K$ that satisfy the Helmholtz equation and admit a factorization 
\begin{equation}
\label{eq:MT86inc}
	w(x,k) = e^{ik \psi(x)} \, w^{\rm slow}(x,k)
\end{equation}
on an open set $O$ containing $K$ where $\psi$ is a smooth phase function having convex wave-fronts
$\{ x : \psi(x) = t \}$ relative to the normal $\nabla \psi$, and $w^{\rm slow}$ is an envelope which belongs
to the H\"{o}rmander class $S^{0}_{1,0}(O \times (0,\infty))$ and admits a classical asymptotic expansion
\begin{equation}
\label{eq:elclassico}
	w^{\rm slow}(x,k) \sim \sum_{p=0}^{\infty} k^{-p} A_{p}(x)
\end{equation}
on the open set $O$.

\end{itemize}
\begin{remark}
Assumption B will be addressed for both the Dirichlet and Neumann conditions in a forthcoming paper. Indeed, as was observed in the setting of the Dirichlet condition \cite{EcevitReitich09}, the stationary points of the \emph{combined phase} (arising in integral equation formulations) display the same geometrical characteristics with those corresponding to a plane wave incidence (see \cite{EcevitReitich09,MelroseTaylor85} for details). It is therefore conceivable that the analysis in \cite{MelroseTaylor85} can be modified to provide a rigorous proof. 
\end{remark}

As was shown in \cite{EcevitReitich09}, the no-occlusion and visibility conditions guarantee
that the multiple scattering phases
\begin{equation} 
\label{eq:multphase}
	\varphi_m(x)= 
	\left\{
		\begin{array}{ll}
			\alpha \cdot x,
			&
			\quad m = 0,
			\\
			\alpha \cdot \mathcal{X}^{m}_{0}(x)
			+ \sum\limits_{j=0}^{m-1} |\mathcal{X}^{m}_{j+1}(x)-\mathcal{X}^{m}_{j}(x)|,
			& \quad m \ge 1 \, ,
		\end{array}
	\right.
\end{equation}
are well defined, for all $m \ge 1$ and all $x \in \partial K_m$, by the conditions
\begin{equation} \label{eq:points}
	\left\{
		\begin{array}{ll}
			\text{(a)} \ (\mathcal{X}^{m}_0(x), \ldots, \mathcal{X}^{m}_{m}(x)) \in \partial K_0 \times \cdots \times \partial K_m,
			\vspace{0.2cm}
			\\
			\text{and}
			\vspace{0.2cm}
			\\
			\text{(b)} \ \alpha \cdot \nu(\mathcal{X}^{m}_{0}(x)) < 0,
			\vspace{0.2cm}
			\\
			\text{(c)} \
			\dfrac{\mathcal{X}^{m}_{1}(x)-\mathcal{X}^{m}_{0}(x)}
			{|\mathcal{X}^{m}_{1}(x)-\mathcal{X}^{m}_{0}(x)|}
			= \alpha - 2 \alpha \cdot \nu(\mathcal{X}^{m}_{0}(x)) \, \nu(\mathcal{X}^{m}_{0}(x)),
			\vspace{0.2cm}
			\\
			\text{and, for } 1 < j < m, 
			\vspace{0.2cm}
			\\
			\text{(d)} \
			(\mathcal{X}^{m}_{j+1}(x)-\mathcal{X}^{m}_{j}(x))
			\cdot \nu(\mathcal{X}^{m}_{j}(x)) > 0,
			\vspace{0.2cm}
			\\
			\text{(e)} \
			\dfrac{\mathcal{X}^{m}_{j+1}(x)-\mathcal{X}^{m}_{j}(x)}
			{|\mathcal{X}^{m}_{j+1}(x)-\mathcal{X}^{m}_{j}(x)|}
			= \dfrac{\mathcal{X}^{m}_{j}(x)-\mathcal{X}^{m}_{j-1}(x)}
			{|\mathcal{X}^{m}_{j}(x)-\mathcal{X}^{m}_{j-1}(x)|}
			\\
			\hspace{3.6cm}
			- 2 \dfrac{\mathcal{X}^{m}_{j}(x)-\mathcal{X}^{m}_{j-1}(x)}
			{|\mathcal{X}^{m}_{j}(x)-\mathcal{X}^{m}_{j-1}(x)|}
			\cdot \nu(\mathcal{X}^{m}_{j}(x)) \, \nu(\mathcal{X}^{m}_{j}(x)),
			\vspace{0.2cm}
			\\
			\text{and}
			\vspace{0.2cm}
			\\
			\text{(f)} \
			\mathcal{X}^{m}_{m}(x) = x.
		\end{array}
	\right.
\end{equation}
Geometrically speaking \eqref{eq:points} means the \emph{broken ray
$(\mathcal{X}^{m}_0(x), \ldots, \mathcal{X}^{m}_{m}(x))$ corresponding to any given
point $x \in \partial K_m$} is determined, for $0 \le j < m-1$, by the \emph{law of reflection}
subject to the condition that the open line segment
$\{ t \mathcal{X}^{m}_j(x) + (1-t) \mathcal{X}^{m}_{j+1}(x) : 0 < t <1 \}$ has no point in
common with $K_j \cup K_{j+1}$.
Consequently, \eqref{eq:multphase} allows for the extraction of the phases of the total fields $\eta_m$ \eqref{eq:multscattotalfields}
in the form
\begin{equation} \label{eq:extractphases}
	\eta_m(x,k) = e^{ik \, \varphi_m(x)} \eta^{\rm slow}_m(x,k),
	\quad
	x \in \partial K_m,
\end{equation}
and the broken rays uniquely partition the boundary $\partial K_m$ into the
\emph{illuminated regions}
\begin{equation}
\label{eq:illum_m_def}
	\partial K^{\rm IL}_{m}
	= \left\{
		\begin{array}{ll}
			\left\{
				x \in \partial K_{0} :
				\alpha \cdot \nu(x) < 0
			\right\},
			& \quad m = 0,
			\\
			\\
			\left\{
				x \in \partial K_{m} :
				(\mathcal{X}^{m}_{m}(x)-\mathcal{X}^{m}_{m-1}(x))
				\cdot \nu(x) < 0
			\right\},
			& \quad m \ge 1,
		\end{array}
	\right.
\end{equation}
\emph{shadow regions}
\begin{equation}
\label{eq:shbound_m_def} 
	\partial K^{\rm SR}_{m}
	= \left\{
		\begin{array}{ll}
			\left\{
				x \in \partial K_{0} :
				\alpha \cdot \nu(x) > 0
			\right\},
			& \quad m = 0,
			\\
			\\
			\left\{
				x \in \partial K_{m} :
				(\mathcal{X}^{m}_{m}(x)-\mathcal{X}^{m}_{m-1}(x))
				\cdot \nu(x) > 0
			\right\}
			& \quad m \ge 1,
		\end{array}
	\right.
\end{equation}
and \emph{shadow boundaries}
\begin{equation}
\label{eq:shreg_m_def}
	\partial K^{\rm SB}_{m}
	= \left\{
		\begin{array}{ll}
			\left\{
				x \in \partial K_{0} :
				\alpha \cdot \nu(x) = 0
			\right\},
			& \quad m = 0,
			\\
			\\
			\left\{
				x \in \partial K_{m} :
				(\mathcal{X}^{m}_{m}(x)-\mathcal{X}^{m}_{m-1}(x))
				\cdot \nu(x) = 0
			\right\}
			& \quad m \ge 1.
		\end{array}
	\right.
\end{equation}
As is apparent from \eqref{eq:inteqall}, the determination of H\"{o}rmander classes and asymptotic expansions of the envelopes
$\eta^{\rm slow}_m$ \eqref{eq:extractphases} further demand a detailed understanding of the asymptotic behavior of the scattered
fields $u_m$. Within this framework, as was shown in \cite{EcevitReitich09}, the phase functions $\varphi_{m}$ \eqref{eq:multphase} admit
smooth and convex wave-fronts
\begin{equation}
\label{eq:mywavefront}
	\mathcal{P}_{m}(t,\mathcal{N}_{m})
	= \left\{
		y + \left( t - \varphi_{m}(y) \right) \alpha^{\rm ref}_{m}(y) \, : \,
		y \in \mathcal{N}_{m}
	\right\}
\end{equation}
for any open connected sub-manifold $\mathcal{N}_{m} \subset \partial K_{m}^{\rm IL}$ ($m \ge 0$) (for all $t$
greater than the minimum of $\varphi_{m}$ on $\mathcal{N}_{m}$) with respect to the normal
\begin{equation}
\label{eq:normalfield}
	\alpha^{\rm ref}_{m}(y)
	= \left\{
		\begin{array}{ll}
			\alpha - 2 \alpha \cdot \nu(y) \, \nu(y) \, ,
			& \quad m = 0,
			\\
			\\
			\dfrac{\mathcal{X}^{m}_{m}(y)-\mathcal{X}^{m}_{m-1}(y)}
			{|\mathcal{X}^{m}_{m}(y)-\mathcal{X}^{m}_{m-1}(y)|}
			- 2 \, \dfrac{\mathcal{X}^{m}_{m}(y)-\mathcal{X}^{m}_{m-1}(y)}
			{|\mathcal{X}^{m}_{m}(y)-\mathcal{X}^{m}_{m-1}(y)|}
			\cdot \nu(y) \, \nu(y) \, ,
			& \quad m \ge 1;
		\end{array}
	\right.
\end{equation}
the normal field $\alpha^{\rm ref}_{m}(y)$ is the direction of the reflected ray resulting from the incidence on $\partial K_{m}$ with direction
\begin{equation*}
	\alpha^{\rm inc}_{m}(y) =
	\left\{
		\begin{array}{ll}
			\alpha,
			& \ m = 0,
			\vspace{0.2cm}
			\\
			\dfrac{\mathcal{X}^{m}_{m}(y)-\mathcal{X}^{m}_{m-1}(y)}
			{|\mathcal{X}^{m}_{m}(y)-\mathcal{X}^{m}_{m-1}(y)|},
			& \ m \ge 1.
		\end{array}
	\right.
\end{equation*}
It follows, upon noting that
\begin{equation}
\label{eq:raym}
	\operatorname{ray}_{m}(y) = \left\{ y + t \alpha^{\rm ref}_{m}(y) : t > 0 \right\}
\end{equation}
is the half-ray generated by the reflection of the ray incident on the boundary $\partial K_{m}$ at $y$ with direction $\alpha^{\rm inc}_{m}(y)$,
for any $x$ in the \emph{open illuminated region}
\begin{equation}
\label{eq:OILm}
	O^{\rm IL}_{m} = \bigcup
	\left\{ \operatorname{ray}_{m}(y) \, : \, y \in \partial K^{\rm IL}_{m} \right\}
\end{equation}
there exists a unique $y = y(x) \in \partial K^{\rm IL}_{m}$ so that
$x \in \operatorname{ray}_{m}(y)$. The \emph{reflected phase function}
\begin{equation}
\label{eq:mainphase}
	\psi_{m}(x) = |x - y(x)| + \varphi_{m}(y(x)),
	\quad
	x \in O^{\rm IL}_{m},
\end{equation}
then has the smooth and convex wave-fronts \eqref{eq:mywavefront}. This, in turn, allows us to extract the phase of the scattered field $u_m$ \eqref{eq:usjdl} in the form
\begin{equation}
\label{eq:umslowdef}
	u_{m}(x,k) = e^{ik \psi_{m}(x)} \, u^{\rm slow}_{m}(x,k) \, ,
	\quad
	x \in O^{\rm IL}_{m}.
\end{equation}
A final note is that the visibility and no-occlusion conditions imply that the obstacle $K_{m+1}$ is contained in the open set $O^{\rm IL}_{m}$ for all $m\ge 0$ \cite{EcevitReitich09}.

Under the assumptions A and B, the asymptotic behavior of the envelopes $\eta_m^{\rm slow}$ is as follows.

\begin{theorem} \label{Thm:main1}
For all $m \in \mathbb{Z}_+ = \mathbb{N} \cup \{ 0 \}$, we have:
\begin{itemize}
\item[\emph{(i)}] $\eta^{\rm slow}_m(x,k) \in S^{0}_{1,0}(\partial K^{\rm IL}_m \times (0,\infty))$ and has an asymptotic expansion
\begin{equation}
\label{eq:MT85a}
	\eta^{\rm slow}_m(x,k)
	\sim \sum_{j \ge 0} k^{-j} a_{m,j}(x) 
\end{equation}
for some complex-valued smooth functions $a_{m,j}$.
\item[\emph{(ii)}] On a $(k$ independent$)$ neighborhood $\partial K^{\rm SB}_{m,\epsilon}$ of 
$\partial K^{\rm SB}_m$, $\eta^{\rm slow}(x,k)$ belongs to $S^0_{\frac{2}{3},\frac{1}{3}}(\partial K^{\rm SB}_{m,\epsilon} \times (0,\infty))$ and has
an asymptotic expansion 
\begin{align}
	\eta^{\rm slow}_m(x,k)
	& \sim
	\sum_{p,q,r \ge 0 \atop \ell \le -1} a_{m,p,q,r,\ell}(x,k)
	\label{eq:MT85mult}
	\\ 
	& = \sum_{p,q,r \ge 0 \atop \ell \le -1} k^{-\frac{1+2p+3q+r+\ell}{3}+(\ell+1)_-} \,
	b_{m,p,q,r,\ell}(x) \, (\Psi^{r,\ell})^{(p)}(k^{\frac{1}{3}}Z_m(x))
	\nonumber
\end{align}
where $b_{m,p,q,r,\ell}$ and $\Psi^{r,\ell}$ are complex-valued smooth functions, 
$Z_m$ is a real-valued smooth function positive on $\partial K^{\rm IL}_m$, negative on $\partial K^{\rm SR}_m$,
and vanishes exactly to first order on $\partial K^{\rm SB}_m$, and $\Psi^{r,\ell}$ are complex-valued smooth functions
with an asymptotic expansion
\begin{equation} \label{eq:Psirlasymp}
	\Psi^{r,\ell}(\tau) \sim \sum_{j=0}^{\infty} a_{r,\ell,j} \tau^{1+\ell-2r-3j}
	\qquad
	\text{as } \tau \to \infty,
\end{equation}
and they rapidly decrease in the sense of Schwartz as $\tau \to -\infty$.
\item[\emph{(iii)}] On any compact subset of $\partial K^{\rm SR}_m$, $\eta^{\rm slow}_m$ rapidly decreases in the sense of Schwartz
as $k \to \infty$.
\item[\emph{(iv)}] Moreover, $\eta^{\rm slow}_m(x,k) \in S^{0}_{\frac{2}{3},\frac{1}{3}}(\partial K_m \times (0,\infty))$ and the
asymptotic expansion \eqref{eq:MT85mult} is valid over the entire boundary $\partial K_m$.
\end{itemize}
\end{theorem}

\begin{proof}
The proof is by induction on $m$.
For $m=0$, the result follows from Theorem~\ref{Thm:MT85}.
For $m \ge 1$, $\eta_{m}$ is the total field generated by the scattered field $u_{m-1}$ impinging on $K_m$ as an incident field
(see \eqref{eq:Helmholtzsigmadecompmnewer} and \eqref{eq:inteqall}). Since $K_m \subset O^{\rm IL}_{m-1}$,
$u_{m-1}(x,k) = e^{ik \psi_{m-1}(x)} \, u^{\rm slow}_{m-1}(x)$ and the phase $\psi_{m-1}$ has convex wave-fronts $\{ \psi_{m-1}(x) = t \}$
in the open set $O^{\rm IL}_{m-1}$, by assumption~A, it is sufficient to prove that
$u^{\rm slow}_{m}(x,k) \in S^{0}_{1,0}(O^{\rm IL}_m \times \left( 0,\infty \right))$ and has a classical asymptotic expansion of the form (\ref{eq:elclassico}).
This technical result is the content of Theorem~\ref{Thm:main2} in the next section.
\end{proof}

As we mentioned in the introduction, 
construction of frequency independent algorithms for multiple scattering problems further
demand the derivation of sharp wavenumber explicit estimates on the derivatives of the multiple
scattering iterations $\eta^{\rm slow}_m$. As in the Dirichlet multiple scattering problem
\cite{EcevitEtAl20}, this is naturally based on the H\"{o}rmander classes
and sharp wavenumber explicit estimates on the derivatives of the envelopes $a_{m,p,q,r,\ell}$
in the Neumann asymptotic expansion \eqref{eq:MT85mult}. In connection therewith, assuming a counterclockwise oriented
$P_m$-periodic regular parameterization $\gamma_m(s)$ of the boundary $\partial K_m$ and
writing $a_{m,p,q,r,\ell}(s,k)$ for $a_{m,p,q,r,\ell}(\gamma_m(s),k)$, we have:

\begin{theorem}
For all $m \in \mathbb{Z}_+$ and $(p,q,r,\ell) \in \mathbb{Z}_+^3 \times (-\mathbb{N})$, $a_{m,p,q,r,l}(s,k)$ belongs to the H\"{o}rmander class
$S^{\vartheta(p,q,r,\ell)}_{\frac{2}{3},\frac{1}{3}} ([0,P_m]\times(0,\infty))$ where
\[
	\vartheta(p,q,r,\ell) = -\frac{1+2p+3q+r+\ell}{3} + (\ell+1)_{-} +
	\left\{
		\begin{array}{ll}
			0, & 1+\ell-2r-p < 0, 
			\\ [0.3 em]
			\frac{1+\ell-2r-p}{3}, & 1+\ell-2r-p \ge 0.
		\end{array}
	\right.
\]
Moreover, for any $k_0 > 0$ and $n \in \mathbb{Z}_+$, we have the estimate
\begin{equation}
\label{eq:apq-estimatecompact}
	| D_s^{n} a_{m,p,q,r,\ell}(s,k)|
	\lesssim
	k^{\vartheta(p,q,r,\ell)} \, W_m(s,k)^{-n},
	\qquad
	\text{for all } (s,k) \in [0,P_m] \times [k_0,\infty),
\end{equation}
where $W_m(s,k) = k^{-\frac{1}{3}} + |\omega_m(s)|$ with $\omega_m(s) = (s-t_{m,1})(t_{m,2}-s)$
and where $\partial K_m^{\rm SB} = \{ \gamma_m(t_{m,1}),\gamma_m(t_{m,2}) \}$.
\end{theorem}
\begin{proof} Follows from an adaptation of the proofs of \cite[Lemmas 2, 3]{EcevitEtAl22} along with \cite[Corollary 1]{EcevitEtAl22}. 

\end{proof}

In the case of single \cite{Ecevit18} and multiple \cite{EcevitEtAl20} scattering 
settings with Dirichlet boundary conditions, it was shown that incorporation of an appropriate
number of terms in the asymptotic expansion \eqref{eq:MTDslow} into integral equation
formulations allows for the development of frequency independent Galerkin boundary element methods.
The same approach was also successfully applied to the case of the Neumann boundary condition
for a single obstacle. For the multiple scattering
problem considered herein, incorporation of the terms $\sigma_{m,\beta}$ into
integral equation formulations transforms the unknown from $\eta_m$ to
$\rho_{m,\beta}$. We introduce these quantities in the following definition.
Throughout the text, we use the standard convention that an empty sum is zero.

\begin{definition} \label{def:quantities}
Given $\beta \in \mathbb{Z}_{+}$, we set
\[
	\mathcal{F}_{\beta} = \left\{ (p,q,r,\ell) \in \mathbb{Z}_{+}^3 \times (-\mathbb{N}) :
	\beta+ 3\vartheta(p,q,r,\ell) > 0 \right\},
\] 
and we define
\begin{equation} \label{eq:sigmarhobetaslow}
	\sigma^{\rm slow}_{m,\beta}(s,k)
	= \!\!\!\!\!\!\! \sum_{(p,q,r,\ell) \in \mathcal{F}_{\beta}} a_{m,p,q,r,\ell}(s,k),
	\qquad
	\rho^{\rm slow}_{m,\beta}(s,k) = \eta^{\rm slow}_m(s,k) - \sigma^{\rm slow}_{m,\beta}(s,k),
\end{equation}
and
\begin{equation} \label{eq:sigmarhobeta}
	\sigma_{m,\beta}(s,k) = e^{ik \, \varphi_m(s)} \sigma^{\rm slow}_{m,\beta}(s,k),
	\qquad
	\rho_{m,\beta}(s,k) = e^{ik \, \varphi_m(s)} \rho^{\rm slow}_{m,\beta}(s,k).
\end{equation}
\end{definition}
Note that $\rho_{m,\beta}(s,k) = \eta_m(s,k) - \sigma_{m,\beta}(s,k)$ and, in particular,
$\rho_{m,0} = \eta_m$.

As in the Neumann single scattering problem \cite{EcevitEtAl22},
sharp explicit estimates on the derivatives of $\rho_{m,\beta}$
with respect to the wavenumber can be used in the design and
analysis of frequency independent Galerkin boundary element
methods for multiple scattering problems. These estimates are summarized in the following.

\begin{theorem} \label{thm:etamder}
For any $m,\beta \in \mathbb{Z}_+$, we have
$\rho_{m,\beta}^{\rm slow}(s,k) \in S^{-\frac{\beta}{3}}_{\frac{2}{3},\frac{1}{3}}([0,P_m] \times (0,\infty))$.
Moreover, for any $k_0 > 1$ and $n \in \mathbb{Z}_+$, we have the estimate
\[
	|D_s^n \rho_{m,\beta}^{\rm slow}(s,k)|
	\lesssim k^{-\frac{\beta}{3}} \, W_m(s,k)^{-n},
	\qquad
	\text{for all } (s,k) \in [0,P_m] \times [k_0,\infty).
\]
\end{theorem}
\begin{proof}
Follows from an adaptation of the proof of \cite[Theorem 3]{EcevitEtAl22}. 

\end{proof}

\section{Asymptotic expansions of the scattered fields}
\label{sect:4}

In this section, we derive the asymptotic expansions of the envelopes $u^{\rm slow}_m$ \eqref{eq:umslowdef}.
This derivation is based on the integral representation
\begin{equation} \label{eq:umslowinrep}
	u^{\rm slow}_m(x,k)
	= \frac{i}{4} \, e^{-ik \psi(x)} k \int_{\partial K_m} e^{ik \, \varphi_m(y)} H^{(1)}_1(k|x-y|) \, \frac{x-y}{|x-y|} \cdot \nu(y) \, \eta^{\rm slow}_m(y,k) \, ds(y)
\end{equation}
which follows from a combination of \eqref{eq:usjdl}, \eqref{eq:extractphases}, and \eqref{eq:umslowdef}. In connection therewith, since
$\eta^{\rm slow}_m$ rapidly decreases in the shadow region, we realize that the illuminated region asymptotic expansion \eqref{eq:MT85a}
of $\eta^{\rm slow}_m$ can be employed in \eqref{eq:umslowinrep} so as to formally have
\begin{equation} \label{eq:umslowinrepapp1}
	u^{\rm slow}_m(x,k)
	\approx \frac{i}{4} \, e^{-ik \psi_m(x)} \!
	\sum_{j \ge 0} k^{1-j} \!\!
	\int_{\partial K^{\rm IL}_m} \!\!\!\!\!\! e^{ik \, \varphi_m(y)} H^{(1)}_1(k|x-y|) \, \frac{x-y}{|x-y|} \cdot \nu(y) \, a_{m,j}(y) \, ds(y).
\end{equation}
To obtain a further approximation, we utilize the asymptotic behavior of Hankel functions in \eqref{eq:umslowinrepapp1}.
In this connection, we recall that if $s,s_1 \in \mathbb{Z}_+$ with $s_1+1 \ge s$ and $k |x-y| \gg 1$ \cite[Eq. 8.451.3]{GradshteynEtAl00}, then
\begin{equation}
\label{eq:Hankelasymptotics}
	H_{s}^{(1)}(k |x-y|) = H_{s,s_1}(k |x-y|) + \overset{\sim}{H}_{s,s_1}(k |x-y|)
\end{equation}
with
\begin{equation}
\label{eq:HankelTip}
	H_{s,s_1}(k |x-y|) = 
	\sum_{s_2=0}^{s_1} \dfrac{e^{ik|x-y|} \, c_{s,s_2}}{(k|x-y|)^{(s_2+\frac{1}{2})}}
\end{equation}
and
\begin{equation}
\label{eq:HankelTail}
	\overset{\sim}{H}_{s,s_1}(k |x-y|)
	= \theta_{s,s_1}(k|x-y|) \, \dfrac{e^{ik|x-y|} \, c_{s,s_1+1}}{(k|x-y|)^{(s_1+\frac{3}{2})}},
\end{equation}
where $|\theta_{s,s_1}(k |x-y|)| < 1$ and with $\Gamma$ denoting the Gamma function
\begin{equation*}
	c_{s,s_2}
	= \dfrac{i^{s_2} \, \Gamma \left( s + s_2 + \frac{1}{2} \right)}
	{\sqrt{\pi} \, e^{\frac{i \pi (2s + 1)}{4}} \, s_2! \, 2^{s_2 -\frac{1}{2}} \, \Gamma \left( s - s_2 + \frac{1}{2} \right)},
	\qquad
	s,s_2 \in \mathbb{Z}_+.
\end{equation*}
Upon using \eqref{eq:Hankelasymptotics} in \eqref{eq:umslowinrepapp1},
we therefore obtain the formal approximation
\begin{multline} \label{eq:umslowinrepapp2}
	u^{\rm slow}_m(x,k)
	\approx \frac{i}{4} \, e^{-ik \psi_m(x)} \!\!
	\sum_{j,s_2 \ge 0} c_{1,s_2} \, k^{\frac{1}{2}-j-s_2} \!\!
	\\
	\int_{\partial K^{\rm IL}_m} \, e^{ik \, (\varphi_m(y)+|x-y|)} \, \frac{x-y}{|x-y|} \cdot \nu(y) \,
	\frac{a_{m,j}(y)}{|x-y|^{s_2+\frac{1}{2}}} \, ds(y).
\end{multline}
To complete the derivation of the asymptotic expansion of $u^{\rm slow}_m$,
we apply the following version of the stationary phase lemma to the integrals on the right-hand side of \eqref{eq:umslowinrepapp2}.

\begin{theorem}[Stationary phase lemma \cite{Fedoryuk71}] \label{thm:spl}
Let $\psi \in C^{\infty}[a,b]$ be real valued, and let $f \in C^{\infty}_{0}[a,b]$.
Suppose that $t_{0}$ is the only stationary point of $\psi$ in $(a,b)$ and that
$\psi''(t_{0}) \ne 0$. Then, for any $N \in \mathbb{N}$,
\begin{equation*}
	\left|
		\int_{a}^{b} e^{ik\psi(t)} \, f(t) \, dt
		- e^{ik\psi(t_{0})} \sum_{q=0}^{N-1} k^{-(q+\frac{1}{2})} \, S_{q}\left[ f(t),\psi(t) \right](t_{0})
	\right| \le c_{N} \, k^{-N} \, \Vert f \Vert_{C^{N+1}[a,b]}
\end{equation*}
holds for $k > 1$. Here, with $\sigma = \operatorname{sign} \psi''(t_{0})$
and $h(t) = \left| t - t_{0} \right| \left[ 4\sigma \left( \psi(t)-\psi(t_{0}) \right) \right]^{-\frac{1}{2}}$,
\begin{equation*}
	S_{q}\left[ f(t),\psi(t) \right]
	= e^{\frac{i\pi \sigma \left( 2q+1 \right)}{4}} \,
	\dfrac{\Gamma \left( q + \frac{1}{2} \right)}{(2q)!} \, 
	\dfrac{d^{2q}}{dt^{2q}} \left[  \left( h(t) \right)^{q + \frac{1}{2}} f(t) \right] .
\end{equation*}
\end{theorem}

To clarify the details of this application, for $x \in O^{\rm IL}_{m}$, we denote by $y(x)$ the unique point in the illuminated region $\partial K^{\rm IL}_{m}$
such that $x \in \operatorname{ray}_{m}(y(x))$, and by $t_x$ the unique point in $[0,P_m)$ with $y(t_x) = y(x)$ where
$y(t)$ is an arc-length parametrization of $\partial K_{m}$. For any $x \in O^{\rm IL}_m$, the only stationary point of the
phase $\varphi_m(y) + |x-y|$ in the illuminated region $\partial K^{\rm IL}_m$ is $y(x)$ and $\varphi_m(y(x)) + |x-y(x)| = \psi(x)$ \cite{EcevitReitich09} so that the stationary phase lemma formally entails the approximation
\[
	u^{\rm slow}_m(x,k)
	\approx
	\sum_{j,s_2,q \ge 0} k^{-j-s_2-q} \, f_{m,j,s_2,q}(x)
\]
with
\begin{equation} \label{eq:fmjs2q}
	f_{m,j,s_2,q}(x)
	= \frac{i}{4}
	c_{1,s_2}
	S_q \!
	\left[ \frac{x-y(t)}{|x-y(t)|} \cdot \nu(y(t)) \frac{a_{m,j}(y(t))}{|x-y(t)|^{s_2+\frac{1}{2}}},
	\varphi_m(y(t)) + |x-y(t)|\right]\!(t_x).
\end{equation}
Grouping together the terms modulated by like powers of $k$, we define for $p \in \mathbb{Z}_+$
\begin{equation} \label{eq:Ampdef}
	A_{m,p}(x) = \sum_{\substack{j,s_2,q \ge 0 \\ j+s_2+q = p}} f_{m,j,s_2,q}(x),
	\qquad
	x \in O^{\rm IL}_m.
\end{equation}
With this definition we now state the main result on the asymptotic expansion of the envelopes $u^{\rm slow}_m$. 

\begin{theorem} \label{Thm:main2}
Assume that Theorem~\ref{Thm:main1} holds for some $m \ge 0$. Then $u^{\rm slow}_{m}(x,k)$
belongs to the H\"{o}ramander class $S^{0}_{1,0}(O^{\rm IL}_{m} \times (0,\infty))$
$($see~\eqref{eq:OILm} and \eqref{eq:umslowdef}$)$ and has an asymptotic expansion
\begin{equation}
\label{eq:Aasy}
	u^{\rm slow}_{m}(x,k) \sim \sum_{p=0}^{\infty} k^{-p} A_{m,p}(x),
	\qquad
	(x,k) \in O^{\rm IL}_m \times (0,\infty),
\end{equation}
where $A_{m,p}$ is as defined in \eqref{eq:Ampdef}. 
\end{theorem}

For the rigorous proof of Theorem~\ref{Thm:main2}, we utilize the following classical result.

\begin{theorem}[Fundamental asymptotic expansion lemma \cite{Hormander66,Hormander71}] \label{thm:fael}
Let $\mathcal{M}$ be a $p$-dimensional $C^{\infty}$ manifold,
$\Gamma$ an open conic subset of $\mathcal{M} \times \mathbb{R}^{q}$, and 
$a_{j} \in S^{\nu_{j}}_{\varrho,\delta}(\Gamma)$ where $\nu_{j} \to - \infty$ as $j \to \infty$.
Let $a \in C^{\infty}(\Gamma)$, and assume that for any compact set
$W \subset \Gamma$ and all multi-indices $\beta,\gamma$
\begin{equation*}
	\left| D^{\beta}_{x} D^{\gamma}_{\xi} a(x,\xi) \right|
	\le C \left( 1 + \left| \xi \right| \right)^{\mu},
	\quad
	\left( x, \xi \right) \in W^{c} = \left\{ \left( x, t\xi \right) \, : \, \left( x, \xi \right) \in W, \, t \ge 1 \right\}
\end{equation*}
holds for some $C$ and $\mu$ depending on $\beta,\gamma$ and $W$.
If there exists $\mu_{N} \to - \infty$ as $N \to \infty$ such that for any compact
set $W \subset \Gamma$ and $N \in \mathbb{Z}_+$
\begin{equation*}
	\left| a(x,\xi) - \sum_{j=0}^{N} a_{j}(x,\xi) \right|
	\le C_{W,N} \left( 1 + \left| \xi \right| \right)^{\mu_{N}},
	\quad
	\left( x, \xi \right) \in W^{c} \, ,
\end{equation*}
it follows that $a \in S_{\varrho,\delta}^{\nu}(\Gamma)$
where $\nu = \sup_{j} \nu_{j}$, and that $a \sim \sum a_{j}$.
\end{theorem}

We shall also need the following estimates.

\begin{lemma} \cite[Lemma 1]{EcevitReitich09} \label{lemma:Psilr}
For all $p,q,r \in \mathbb{Z}_+$ and $\ell \in \mathbb{Z}$, the estimates
\begin{equation} \label{eq:psi-derivatives}
	|(\Psi^{r,\ell})^{(p)}(\tau)|
	\lesssim
	\left\{ \!\!
		\begin{array}{ll}
			(1 + |\tau|)^{\gamma_{r,\ell}-p},
			& \text{if } \ p > 1+\ell -2r \ge 0,
			\\
			(1+|\tau|)^{1+\ell-2r-p},
			& \text{otherwise},
		\end{array}
	\right.
\end{equation}
hold for all $\tau \in \mathbb{R}$ where
\[
	1+\ell -2r \equiv \gamma_{r,\ell} \mod 3
	\qquad
	\text{with}
	\qquad
	\gamma_{r,\ell} \in \{ -3,-2,-1\}.
\]
\end{lemma}

\noindent
\emph{Proof of Theorem~\ref{Thm:main2}} \
Given a compact set $\mathcal{S} \subset O^{\rm IL}_{m}$, the visibility and no-occlusion
conditions imply that there exists an $\varepsilon >0$ (depending only on $\mathcal{S}$)
such that $\mathcal{S}$ is contained in the open set 
\begin{equation}\label{eq:OILmeps}
	O^{\rm IL}_{m,3\varepsilon}
	:= \bigcup
	\left\{
		\operatorname{ray}_{m}(y) \, : \,
		y \in \partial K^{\rm IL}_{m} \backslash \partial K^{\rm SB}_{m,3\varepsilon}
	\right\}
\end{equation}
where
\begin{equation} \label{eq:SBeps}
	\partial K_{m,\varepsilon}^{\rm SB}
	:= \left\{
		y \in \partial K_{m} : \operatorname{dist} \left( y, \partial K_{m}^{\rm SB} \right) \le \varepsilon 
	\right\}.
\end{equation}
Introduce a smooth partition of unity $\{ \rho_{1}, \rho_{2}, \rho_3 \}$ on $\partial K_{m}$ such that
\begin{equation}
\label{eq:defparunity}
	\left\{ \!\!\!
		\begin{tabular}{ll}
			(i) & \!\!\!\!\!$\rho_{1} = 1$ on
			$\partial K^{\rm IL}_{m} \backslash \partial K^{\rm SB}_{m,2\varepsilon}$
			and $\rho_{1} = 0$ on
			$\partial K^{\rm SR}_{m} \cup \partial K^{\rm SB}_{m,\varepsilon}$,
			\\
			(ii) & \!\!\!\!\!$\rho_{3} = 1$ on
			$\partial K^{\rm SR}_{m} \backslash \partial K^{\rm SB}_{m,2\varepsilon}$
			and $\rho_{3} = 0$ on
			$\partial K^{\rm IL}_{m} \cup \partial K^{\rm SB}_{m,\varepsilon}$,
		\end{tabular}
	\right.
\end{equation}
and use \eqref{eq:umslowinrep} to write
\begin{multline} \label{eq:umslow}
	u^{\rm slow}_{m}(x,k)
	= \frac{ik}{4} \sum_{j=1}^{3} I_{j}(x,k)
	\\
	= \frac{ik}{4} \sum_{j=1}^{3}
	e^{-ik\psi_{m}(x)} \!\!\! 
	\int\limits_{\partial K_{m}} \! \rho_j(y) \, H^{(1)}_1(k|x-y|) \, \frac{x-y}{|x-y|} \cdot \nu(y) \, e^{ik \, \varphi_m(y)} \eta^{\rm slow}_{m}(y,k) \, ds(y).
\end{multline}
For clarity, we divide the proof into three parts.

\noindent
{\bf Part 1:} Here we show for all $k_0 > 0$, $\zeta \in \mathbb{Z}^2_+$, and $n \in \mathbb{Z}_+$ that
\begin{equation}
\label{eq:rapiddecaynew} 
	\left| D^{\zeta}_{x} D^{n}_{k} u^{\rm slow}_m(x,k) \right|
	\le C_{\mathcal{S},\zeta,n} \left( 1 + k \right)^{\mu_{\mathcal{S},\zeta,n}},
	\qquad
	(x,k) \in \mathcal{S} \times (k_{0},\infty),
\end{equation}
where $\mu_{\mathcal{S},\zeta,n} = 2|\zeta| + \frac{1}{2}$.

\noindent
{\bf Part 1a:} First we show for $j \in \{1,2,3 \}$, $\zeta \in \mathbb{Z}^2_+$, and $n \in \mathbb{Z}_+$ that
\begin{multline} \label{eq:Ijderivativecompactlast}
	D^{\zeta}_{x} D^{n}_{k} \, I_{j}(x,k)
	= e^{-ik \psi_{m}(x)} \!\!\!
	\sum_{\substack{0 \le n_1 \le 2|\zeta| \\ 0 \le n_2 \le n \\ 0 \le n_3 \le n+|\zeta|}} \!\!\!
	k^{n_1} \!
	\int_{\partial K_{m}} \!\!\!
	e^{ik \varphi_{m}(y)} \,
	H^{(1)}_{n_2}(k|x-y|) \,
	D^{n_3}_{k} \left[ \eta_{m}^{\rm slow}(y,k) \right]
	\\
	F_{j,\zeta,n,n_1,n_2,n_3}(x,y) \,
	ds(y)
\end{multline}
where $F_{j,\zeta,n,n_1,n_2,n_3}$ is smooth on
$O^{\rm IL}_{m} \times \partial K_{m}$ and
$\operatorname{supp} F_{j,\zeta,n,n_1,n_2,n_3}(x,\cdot) \subset \operatorname{supp} \rho_{j}(\cdot)$
for all $x \in O^{\rm IL}_{m}$. 
To this end, we apply the multivariate Leibniz's rule for the derivatives of products
to have
\begin{multline*}
	D^{\zeta}_{x} D^{n}_{k} \, I_{j}(x,k)
	= \sum_{\zeta_2 \le \zeta_1 \le \zeta \atop 0 \le n_2 \le n_1 \le n}
	\binom{\zeta}{\zeta_1} \binom{\zeta_1}{\zeta_2}
	\binom{n}{n_1} \binom{n_1}{n_2} \, D^{\zeta-\zeta_1}_{x} D^{n-n_1}_{k} \left[ e^{-ik \psi_{m}(x)} \right]
	\\
	\times \int_{\partial K_{m}} \!\!\!
	\rho_{j}(y) \,
	D^{\zeta_1-\zeta_2}_{x} D^{n_1-n_2}_{k}
	\left[ H^{(1)}_{1}(k|x-y|) \right]
	\\
	D^{\zeta_2}_{x} \left[ \frac{x-y}{|x-y|} \cdot \nu(y) \right]
	D^{n_2}_{k}
	\left[ e^{ik \, \varphi_m(y)} \eta^{\rm slow}_{m}(y,k) \right]
	ds(y).
\end{multline*}
Invoking the derivatives with respect to $k$, we get
\begin{multline*}
	D^{\zeta}_{x} D^{n}_{k} \, I_{j}(x,k)
	= \!\!\!\!\!\!\!\!\!\!\!\! \sum_{\zeta_2 \le \zeta_1 \le \zeta \atop 0 \le n_3 \le n_2 \le n_1 \le n}
	\!\!\!\!\!\!\!\!\!\!\!
	\binom{\zeta}{\zeta_1} \!
	\binom{\zeta_1}{\zeta_2} \!
	\binom{n}{n_1} \!
	\binom{n_1}{n_2} \!
	\binom{n_2}{n_3} \!
	D^{\zeta-\zeta_1}_{x} \left[ e^{-ik \psi_{m}(x)} \, \left( -i\psi_{m}(x) \right)^{n-n_1} \right]
	\\
	\hspace{0.3cm}\times \int_{\partial K_{m}} \!\!\!
	\rho_{j}(y) \,
	D^{\zeta_1-\zeta_2}_{x}
	\left[ \left( \! H^{(1)}_{1} \! \right)^{\!\!(n_1-n_2)} (k|x-y|) \left| x-y \right|^{n_1-n_2}\right]
	D^{\zeta_2}_{x} \left[ \frac{x-y}{|x-y|} \cdot \nu(y) \right]
	\\
	e^{ik \varphi_{m}(y)} 
	\left[ i \, \varphi_{m}(y) \right]^{n_2-n_3} 
	D^{n_3}_{k} \left[ \eta_{m}^{\rm slow}(y,k) \right]
	ds(y).
\end{multline*}
Using Leibniz's rule once more, we therefore obtain
\begin{multline*}
	D^{\zeta}_{x} D^{n}_{k} \, I_{j}(x,k)
	= \!\!\!\!\!\!\!\!\!\!\!\!
	\sum_{\zeta_2 \le \zeta_1 \le \zeta \atop 0 \le n_3 \le n_2 \le n_1 \le n}
	\sum_{\zeta_3 \le \zeta - \zeta_1 \atop \zeta_4 \le \zeta_1 - \zeta_2}
	\binom{\zeta}{\zeta_1} \!
	\binom{\zeta_1}{\zeta_2} \!
	\binom{\zeta-\zeta_1}{\zeta_3} \!
	\binom{\zeta_1-\zeta_2}{\zeta_4} \!
	\binom{n}{n_1} \!
	\binom{n_1}{n_2} \!
	\binom{n_2}{n_3} \!
	\\
	\times
	D^{\zeta-\zeta_1-\zeta_3}_{x}
	\left[ e^{-ik \psi_{m}(x)} \right] D^{\zeta_3}_{x} \left[ \left( -i\psi_{m}(x) \right)^{n-n_1} \right]
	\\
	\times \int_{\partial K_{m}} \!\!\!
	\rho_{j}(y) \,
	D^{\zeta_1-\zeta_2-\zeta_4}_{x}
	\left[ \left( \! H^{(1)}_{1} \! \right)^{\!\!(n_1-n_2)} (k|x-y|) \right]
	D^{\zeta_4}_{x} \left[ \left| x-y \right|^{n_1-n_2}\right]
	\\
	D^{\zeta_2}_{x} \left[ \frac{x-y}{|x-y|} \cdot \nu(y) \right]
	e^{ik \varphi_{m}(y)} 
	\left[ i \, \varphi_{m}(y) \right]^{n_2-n_3} 
	D^{n_3}_{k} \left[ \eta_{m}^{\rm slow}(y,k) \right]
	ds(y),
\end{multline*}
and we rewrite this as
\begin{multline} \label{eq:dersofIjk}
	D^{\zeta}_{x} D^{n}_{k} \, I_{j}(x,k)
	= \sum_{\zeta_2 \le \zeta_1 \le \zeta \atop 0 \le n_3 \le n_2 \le n_1 \le n}
	\sum_{\zeta_3 \le \zeta - \zeta_1 \atop \zeta_4 \le \zeta_1 - \zeta_2}
	D^{\zeta-\zeta_1-\zeta_3}_{x} \left[ e^{-ik \psi_{m}(x)} \right]
	\\
	\times \int_{\partial K_{m}} \!\!\!
	\rho_{j}(y) \,
	e^{ik \varphi_{m}(y)} \,
	D^{\zeta_1-\zeta_2-\zeta_4}_{x} \left[ \left( \! H^{(1)}_{1} \! \right)^{\!\!(n_1-n_2)} (k|x-y|) \right]
	D^{n_3}_{k} \left[ \eta_{m}^{\rm slow}(y,k) \right]
	\\
	F_{\zeta,\zeta_1,\zeta_2,\zeta_3,\zeta_4,n,n_1,n_2,n_3}(x,y) \,
	ds(y).
\end{multline}
Next we use multivariate Fa\`{a} di Bruno formula~\cite{ConstantineSavits96} for the
derivatives of compositions which entails for $\zeta' = \zeta-\zeta_1-\zeta_3 \in \mathbb{Z}^2_+ \setminus \{ (0,0) \}$
\begin{multline*}
	D_{x}^{\zeta'} \left[ e^{-ik \psi_{m}(x)} \right]
	= e^{-ik \psi_{m}(x)}
	\\
	\times \left( \zeta'!
	\sum_{1 \le n_4, n'_4 \le |\zeta'|} \left( -k \right)^{n_4} \,
	\sum_{FdB(\zeta';n_4,n'_4)}
	\prod_{1 \le j \le n'_4} \dfrac{1}{\ell_{j}!}
	\left( \dfrac{D^{\mu_{j}}_{x} \left[ \psi_{m}(x) \right]}{\mu_{j}!} \right)^{\ell_{j}} \right)
\end{multline*}
and for $\zeta' = \zeta_1-\zeta_2-\zeta_4 \in \mathbb{Z}^2_+ \setminus \{ (0,0) \}$ and $n' = n_1-n_2 \in \mathbb{Z}_+$
\begin{multline*}
	D_{x}^{\zeta'} \left[ \left( \! H^{(1)}_{1} \! \right)^{\!\!(n')} \!\!\! \left( k \left| x-y \right| \right) \right]
	= \zeta'! \!\!\!\!\!\! \sum_{1 \le n_5, n'_5 \le |\zeta'|}
	\!\!\! \left( \! H^{(1)}_{1} \! \right)^{\!\!(n'+n_5)} \!\!\! \left( k \left| x-y \right| \right) k^{n_5} \!\!\!
	\!\!\!\! 
	\\
	\sum_{FdB(\zeta';n_5,n'_5)}
	\prod_{1 \le j \le n'_5} \dfrac{1}{\ell_{j}!}
	\left( \dfrac{D^{\mu_{j}}_{x} \left| x-y \right|}{\mu_{j}!} \right)^{\ell_{j}}
\end{multline*}
where, in general, for $\zeta \in \mathbb{Z}_+ \setminus \{ (0,0) \}$ and $n_1,n_2 \in \{1,\ldots, |\zeta| \}$
\begin{multline*}
	FdB(\zeta;n_1,n_2)
	= \Bigg\{ \left( \ell_{1},\ldots, \ell_{n_2};\mu_{1},\ldots, \mu_{n_2} \right) \in \mathbb{N}^{n_2} \times (\mathbb{Z}_+^2)^{n_2} :
	\\
	\sum_{j=1}^{n_2} \ell_{j} = n_1, \quad \sum_{j=1}^{n_2} \ell_{j} \mu_{j} = \zeta, \quad (0,0) \prec \mu_{1} \prec \ldots \prec \mu_{n_2} \Bigg\},
\end{multline*}
and where the notation $\mu_j \prec \mu_{j'}$ means that if $\mu_j = (\mu_{j,1},\mu_{j,2})$ and $\mu_{j'} = (\mu_{j',1},\mu_{j',2} )$, then
\[
	\left[ |\mu_j| < |\mu_{j'}| \right]
	\quad \text{or} \quad
	\left[ |\mu_j| = |\mu_{j'}| \ \text{and} \ \mu_{j,1} < \mu_{j',1} \right]
\]
or
\[
	\left[ |\mu_j| = |\mu_{j'}| \ \text{and} \ \mu_{j,1} = \mu_{j',1} \ \text{and} \ \mu_{j,2} < \mu_{j',2} \right] .
\]
We therefore deduce for $\zeta' = \zeta-\zeta_1-\zeta_3 \in \mathbb{Z}^2_+$
\begin{equation} \label{eq:eikpsider}
	D_{x}^{\zeta'} \left[ e^{-ik \psi_{m}(x)} \right]
	= e^{-ik \psi_{m}(x)}
	\sum_{0 \le n_4 \le |\zeta'|} k^{n_4} \,
	F_{\zeta',n_4}(x),
\end{equation}
and for $\zeta' = \zeta_1-\zeta_2-\zeta_4 \in \mathbb{Z}^2_+$ and $n' = n_1-n_2 \in \mathbb{Z}_+$
\begin{equation} \label{eq:Hankeldxdk}
	D_{x}^{\zeta'} \left[ \left( \! H^{(1)}_{1} \! \right)^{\!\!(n')} \!\!\! \left( k \left| x-y \right| \right) \right]
	= \sum_{0\le n_5 \le |\zeta'|}
	\!\!\! \left( \! H^{(1)}_{1} \! \right)^{\!\!(n'+n_5)} \!\!\! \left( k \left| x-y \right| \right) k^{n_5} \,
	F_{\zeta',n_5}(x,y).
\end{equation}
Using \eqref{eq:eikpsider} and \eqref{eq:Hankeldxdk} in \eqref{eq:dersofIjk}, we get
\begin{multline*}
	D^{\zeta}_{x} D^{n}_{k} \, I_{j}(x,k)
	= e^{-ik \psi_{m}(x)} \!\!\!\!\!\
	\sum_{\zeta_2 \le \zeta_1 \le \zeta \atop 0 \le n_3 \le n_2 \le n_1 \le n}
	\sum_{\zeta_3 \le \zeta - \zeta_1 \atop \zeta_4 \le \zeta_1 - \zeta_2}
	\sum_{0 \le n_4 \le |\zeta-\zeta_1-\zeta_3| \atop 0 \le n_5 \le |\zeta_1-\zeta_2-\zeta_4|}
	k^{n_4+n_5} \,
	\\
	\times \int_{\partial K_{m}} \!\!\!
	\rho_{j}(y) \,
	e^{ik \varphi_{m}(y)} \,
	\left( H^{(1)}_{1} \! \right)^{\!\!(n_1-n_2+n_5)} (k|x-y|) \,
	D^{n_3}_{k} \left[ \eta_{m}^{\rm slow}(y,k) \right]
	\\
	F_{\zeta,\zeta_1,\zeta_2,\zeta_3,\zeta_4,n,n_1,n_2,n_3,n_4,n_5}(x,y) \,
	ds(y).
\end{multline*}
Upon noting that \cite[Equations 9.1.31 and 9.1.6]{AbramowitzStegun64}
\[
	D^{n}_{z} \left[ H^{(1)}_{1}(z) \right]
	= \dfrac{1}{2^{n}} \sum_{n_1=0}^{n} \binom{n}{n_1} \left( -1 \right)^{n_1} H^{(1)}_{1-n+2n_1}(z)
	\quad
	\text{and}
	\quad
	H^{(1)}_{-n}(z)
	= e^{i \pi n} H^{(1)}_{n}(z)
\]
for any $n \in \mathbb{Z}_+$, we therefore obtain
\begin{multline} \label{eq:Ijderivative}
	D^{\zeta}_{x} D^{n}_{k} \, I_{j}(x,k)
	= e^{-ik \psi_{m}(x)} \!\!\!\!\!\!\!\!
	\sum_{\zeta_2 \le \zeta_1 \le \zeta \atop 0 \le n_3 \le n_2 \le n_1 \le n}
	\sum_{\zeta_3 \le \zeta - \zeta_1 \atop \zeta_4 \le \zeta_1 - \zeta_2}
	\sum_{0 \le n_4 \le |\zeta-\zeta_1-\zeta_3| \atop 0 \le n_5 \le |\zeta_1-\zeta_2-\zeta_4|}
	\sum_{0 \le n_6 \le 1+ n_1-n_2+n_5}
	\!\!\!\!\!\!
	k^{n_4+n_5}
	\\
	\hspace{2.9cm}
	\times \int_{\partial K_{m}} \!\!\!
	e^{ik \varphi_{m}(y)} \,
	H^{(1)}_{n_6}(k|x-y|) \,
	D^{n_3}_{k} \left[ \eta_{m}^{\rm slow}(y,k) \right]
	\\
	F_{j,\zeta,\zeta_1,\zeta_2,\zeta_3,\zeta_4,n,n_1,n_2,n_3,n_4,n_5,n_6}(x,y) \,
	ds(y)
\end{multline}
where $F_{j,\zeta,\zeta_1,\zeta_2,\zeta_3,\zeta_4,n,n_1,n_2,n_3,n_4,n_5,n_6}$ is smooth on
$O^{\rm IL}_{m} \times \partial K_{m}$ and
\[
	\operatorname{supp} F_{j,\zeta,\zeta_1,\zeta_2,\zeta_3,\zeta_4,n,n_1,n_2,n_3,n_4,n_5,n_6}(x,\cdot) \subset \operatorname{supp} \rho_{j}(\cdot)
\]
for all $x \in O^{\rm IL}_{m}$. Therefore, with obvious identifications, \eqref{eq:Ijderivativecompactlast} follows from \eqref{eq:Ijderivative}.

\noindent
{\bf Part 1b:}
Here we show for all $k_{0} \in (0,\infty)$, $\zeta \in \mathbb{Z}^2_+$, and $n \in \mathbb{Z}_+$,
\begin{equation} \label{eq:rapiddecaynew1}
	\left| D^{\zeta}_{x}D^{n}_{k} I_{1}(x,k) \right| \le C_{S,\zeta,n} \left( 1 + k \right)^{2|\zeta|-\frac{1}{2}},
	\qquad
	(x,k) \in S \times (k_{0},\infty).
\end{equation}
Since the left- and right-hand sides of \eqref{eq:rapiddecaynew1} depend continuously on $k$ and  $\mathcal{S}$ is compact,
we need to prove \eqref{eq:rapiddecaynew1} only for $k_{0} \gg 1$. 
Since the compact sets $\mathcal{S}$ and $\partial K_m$ are disjoint, we have $\operatorname{dist}(\mathcal{S},\partial K_m) > 0$,
and therefore we may assume that $k_{0}$ is sufficiently large so that, for all $k > k_{0}$, \eqref{eq:Hankelasymptotics} is
satisfied for all $(x,y) \in \mathcal{S} \times \partial K_{m}$. However, in this case, \eqref{eq:Ijderivativecompactlast} implies
\begin{multline*}
	\left| D^{\zeta}_{x} D^{n}_{k} \, I_{1}(x,k) \right| 
	\le \!\! \sum_{\substack{0 \le n_1 \le 2|\zeta| \\ 0 \le n_2 \le n \\ 0 \le n_3 \le n+|\zeta|}} \!\!\!
	k^{n_1} \!
	\int_{\partial K_{m}} \!
	\left| H^{(1)}_{n_2}(k|x-y|) \right| \,
	\left| D^{n_3}_{k} \left[ \eta_{m}^{\rm slow}(y,k) \right] \right|
	\\
	\left| F_{1,\zeta,n,n_1,n_2,n_3}(x,y) \right|
	ds(y)
\end{multline*}
so that \eqref{eq:rapiddecaynew1} follows from \eqref{eq:Hankelasymptotics}-\eqref{eq:HankelTip}-\eqref{eq:HankelTail}
and $\eta^{\rm slow}_m \in S^0_{\frac{2}{3},\frac{1}{3}}(\partial K_m \times(0,\infty))$.

\noindent
{\bf Part 1c:}
For $j=2,3$, here we prove that $I_{j} \in S^{-\infty}_{1,0}(O^{\rm IL}_{m,3\varepsilon} \times (0,\infty))$.
For this, we have to show, for any compact set $S \subset O^{\rm IL}_{m,3\varepsilon}$,
$k_{0} \in (0,\infty)$, $\zeta \in \mathbb{Z}^2_+$, and $n,N \in \mathbb{Z}_+$,
\begin{equation}
\label{eq:rapiddecay}
	\left| D^{\zeta}_{x}D^{n}_{k} I_{j}(x,k) \right| \le C_{S,N,\zeta,n} \left( 1 + k \right)^{-N},
	\qquad
	(x,k) \in S \times (k_{0},\infty).
\end{equation}
Reasoning as in Part 1b, we deduce that it is sufficient to prove \eqref{eq:rapiddecay} only for $k_{0} \gg 1$. More precisely, we may assume that
$k_{0}$ is sufficiently large so that, for all $k > k_{0}$, \eqref{eq:Hankelasymptotics} is satisfied for all $(x,y) \in S \times \partial K_{m}$.

With this assumption, we now prove \eqref{eq:rapiddecay}. In connection therewith, \eqref{eq:Ijderivative}
implies via triangle inequality the sufficiency of establishing, for any smooth function $F_j: O^{\rm IL}_{m} \times \partial K_{m} \to \mathbb{R}$
with $\operatorname{supp} F_j(x,\cdot) \subset \rho_j(\cdot)$ for all $x \in O^{\rm IL}_{m}$,  the estimates   
\begin{equation}
\label{eq:simplifyme}
	\left| I_{F_j,n,s}(x,k) \right| \le C_{S,N,F_j,n,s} \left( 1 + k \right)^{-N},
	\qquad
	(x,k) \in S \times (k_{0},\infty),
\end{equation}
for all $n,s,N \in \mathbb{Z}_+$ where, for $(x,k) \in O^{\rm IL}_{m} \times (0,\infty)$,
\[
	I_{F_j,n,s}(x,k) = 
	\int_{\partial K_{m}}
	e^{ik \varphi_{m}(y)} \,
	H^{(1)}_{s}(k|x-y|) \
	D^{n}_{k} \left[ \eta^{\rm slow}_{m}(y,k) \right]
	F_j(x,y) \, ds(y).
\]

Statement (\ref{eq:simplifyme}), in turn, will follow provided we prove that, for all $\beta,n,s \in \mathbb{Z}_+$,
\begin{equation}
\label{eq:simplifyfurther}
	\left| I_{\beta,F_{j},n,s}(x,k) \right| \le C_{S,N,\beta,F_{j},n,s} \left( 1 + k \right)^{-N},
	\qquad
	(x,k) \in S \times (k_{0},\infty)
\end{equation}
holds for all $N \in \mathbb{Z}_+$
where, for $(x,k) \in O^{\rm IL}_{m} \times (0,\infty)$,
\begin{equation*}
	I_{\beta,F_{j},n,s}(x,k) = 
	\int_{\partial K_{m}}
	e^{ik \varphi_{m}(y)} \,
	H_{s,\beta}(k|x-y|) \,
	D^{n}_{k} \left[ \sigma_{m,\beta}^{\rm slow}(y,k) \right]
	F_{j}(x,y) \, ds(y).
\end{equation*}
Indeed, using \eqref{eq:Hankelasymptotics}-\eqref{eq:HankelTip}-\eqref{eq:HankelTail} and
$\eta_{m}^{\rm slow} = \sigma_{m,\beta}^{\rm slow} + \rho_{m,\beta}^{\rm slow}$,
we get that, for all $N,n,s,\beta  \in \mathbb{Z}_+$ with $\beta+1 \ge s$ and all $(x,k) \in S \times (k_0,\infty)$,
\begin{multline*}
	\left| I_{F_{j},n,s}(x,k) - I_{\beta,F_{j},n,s}(x,k) \right|
	\\
	\le
	\int_{\partial K_{m}}
	\left| H_{s,\beta}(k|x-y|) \right|
	\left| D^{n}_{k} \left[ \rho_{m,\beta}^{\rm slow}(y,k) \right] \right|
	\left| F_{j}(x,y) \right|
	ds(y)
	\\
	+ \int_{\partial K_{m}}
	\left| \overset{\sim}{H}_{s,\beta}(k|x-y|) \right|
	\left| D^{n}_{k} \left[ \eta_{m}^{\rm slow}(y,k) \right] \right|
	\left| F_{j}(x,y) \right|
	ds(y);
\end{multline*}
(\ref{eq:HankelTip}) and (\ref{eq:HankelTail}), in turn, imply
\begin{multline*}
	\left| I_{F_{j},n,s}(x,k) - I_{\beta,F_{j},n,s}(x,k) \right|
	\le C_{S,\beta,F_{j},s} \, (1+k)^{-\frac{1}{2}}
	\int_{\partial K_{m}}
	\left| D^{n}_{k} \left[ \rho_{m,\beta}^{\rm slow}(y,k) \right] \right|
	ds(y)
	\\
	+ C_{S,\beta,F_{j},s} \, (1+k)^{-(\beta + \frac{3}{2})}
	\int_{\partial K_{m}}
	\left| D^{n}_{k} \left[ \eta_{m}^{\rm slow}(y,k) \right] \right|
	ds(y);
\end{multline*}
using $\rho_{m,\beta}^{\rm slow} \in S^{-\frac{\beta}{3}}_{\frac{2}{3},\frac{1}{3}}$
and $\eta_{m}^{\rm slow} \in S^{0}_{\frac{2}{3},\frac{1}{3}}$, we therefore obtain
\begin{align*}
	\Big| I_{F_{j},n,s}(x,k) & - I_{\beta,F_{j},n,s}(x,k) \Big|
	\\
	& \le C_{S,\beta,F_{j},n,s} \left( (1+k)^{-\frac{1}{2}} (1+k)^{-\frac{\beta}{3}-\frac{2n}{3}}
	+ (1+k)^{-(\beta + \frac{3}{2})} (1+k)^{-\frac{2n}{3}}  \right)
	\\
	& \le C_{S,\beta,F_{j},n,s} \, (1+k)^{-(\frac{1}{2}+\frac{\beta}{3}+\frac{2n}{3})}
	\\
	& \le C_{S,N,\beta,F_{j},n,s} \, (1+k)^{-N}	
\end{align*}
for all $k > k_{0}$ provided $\beta \ge \max \{s+1, 3N-2n-1 \}$. By triangle inequality, this justifies the sufficiency of proving \eqref{eq:simplifyfurther}.

In light of \eqref{eq:HankelTip}, we see that statement \eqref{eq:simplifyfurther} will follow provided we prove,
 for all $\beta,n \in \mathbb{Z}_+$, that
\begin{equation}
\label{eq:simplifyfurther1}
	\left| I_{\beta,F_{j},n}(x,k) \right| \le C_{S,N,\beta,F_{j},n} \left( 1 + k \right)^{-N},
	\qquad
	(x,k) \in S \times (k_{0},\infty)
\end{equation}
holds for all $N \in \mathbb{Z}_+$
where, for $(x,k) \in O^{\rm IL}_{m} \times (0,\infty)$,
\begin{equation*}
	I_{\beta,F_{j},n}(x,k) = 
	\int_{\partial K_{m}}
	e^{ik (|x-y|+\varphi_{m}(y))} 
	D^{n}_{k} \left[ \sigma_{m,\beta}^{\rm slow}(y,k) \right]
	F_{j}(x,y) \, ds(y).
\end{equation*}

Since $\sigma_{m,\beta}^{\rm slow}$ is a finite sum of $a_{m,p,q,r,\ell}$, 
\begin{multline*}
	D_k^n \left[ a_{m,p,q,r,\ell}(s,k) \right]
	\\
	= \sum_{n_1 = n_0}^n  \binom{n}{n_1} k^{n_1-n-\frac{1+2p+3q+r+\ell}{3}+(\ell+1)_-} \, b_{m,p,q,r,\ell}(s) \,
	D_k^{n_1} \left[ (\Psi^{r,\ell})^{(p)}(k^{\frac{1}{3}}Z_m(s)) \right]	
\end{multline*}
(where $n_0$ is $n$ or $0$ depending respectively on the condition that $-\frac{1+2p+3q+r+\ell}{3}+(\ell+1)_-$ is $0$ or negative),
and the single variable Fa\`{a} di Bruno's formula for the derivatives of a composition~\cite{Johnson02} entails for $n_1 \in \mathbb{N}$
\begin{equation*}
	D_{\tau}^{n_1} \left[ (\Psi^{r,\ell})^{(p)}(k^{\frac{1}{3}}Z(\tau)) \right]
	= \!\!\!\!\! \sum_{\mu \, \cdot \, \xi(n_1) \, = \, n_1}
	k^{\frac{|\mu|}{3}} \,
	(\Psi^{r,\ell})^{(p+|\mu|)}(k^{\frac{1}{3}}Z(\tau))
	\prod_{j = 1}^{n_1} \dfrac{j}{\mu_{j}!}
	\left( \dfrac{D_\tau^j Z(\tau)}{j!} \right)^{\mu_{j}}
\end{equation*}
where $\xi(n_1) = (1,\ldots,n_1)$ and $\mu = (\mu_{1}, \ldots, \mu_{n_1})$ is any multi-index, 
statement \eqref{eq:simplifyfurther1} will follow provided we prove,
for all $p,r  \in \mathbb{Z}_+$ and $\ell \in - \mathbb{N}$, that
\begin{equation}
\label{eq:simplifyfurther2}
	\left| I_{F_{j},p,r,\ell}(x,k) \right| \le C_{S,N,F_{j},p,r,\ell} \left( 1 + k \right)^{-N},
	\qquad
	(x,k) \in S \times (k_{0},\infty)
\end{equation}
holds for all $N \in \mathbb{Z}_+$
where, for $(x,k) \in O^{\rm IL}_{m} \times (0,\infty)$,
\begin{equation*}
	I_{F_{j},p,r,\ell}(x,k) = 
	\int_{\partial K_{m}}
	e^{ik (|x-y|+ \varphi_{m}(y))} \,
	(\Psi^{r,\ell})^{(p)}(k^{\frac{1}{3}}Z(y)) \,
	F_{j}(x,y) \, ds(y).
\end{equation*}
For $j = 3$, \eqref{eq:simplifyfurther2} follows from
\begin{align*}
	\left| I_{F_{3},p,r,\ell}(x,k) \right| 
	& \le \int_{\partial K_{m}}
	\left| (\Psi^{r,\ell})^{(p)}(k^{\frac{1}{3}}Z(y)) \right|
	\left| F_{3}(x,y) \right| ds(y)
	\le C_{S,N,F_{3},p,r,\ell} \left( 1 + k \right)^{-N}
\end{align*}
where the last inequality is a consequence of the facts that $Z_m$ is negative on the shadow region $\partial K_m^{\rm SR}$,
the support of $F_3(x,\cdot)$ is a compact subset of $\partial K_m^{\rm SR}$ for all $x \in O_m^{\rm IL}$, and $\Psi^{r,\ell}(\tau)$
decreases rapidly in the sense of Schwartz as $\tau \to -\infty$. As for $j=2$, we note that the phase $\lambda_x(y) = |x-y| + \varphi_m(y)$
has no stationary point in the support of $\rho_2$, and therefore switching to parametric form, repeated integration by parts yields
(writing, with abuse of notation, $\tau$ for $y(\tau)$) for all $N_0 \in \mathbb{Z}_+$
\begin{align}
	I_{F_{2},p,r,\ell}(x,k)
	& = 
	\int_0^{P_m}
	e^{ik \, \lambda_x(\tau)} \,
	(\Psi^{r,\ell})^{(p)}(k^{\frac{1}{3}}Z(\tau)) \,
	F_{2}(x,\tau) \, d\tau
	\nonumber
	\\
	& = \left( \dfrac{i}{k} \right)^{N_0} 
	\int_0^{P_m}
	e^{ik \, \lambda_x(\tau)} \,
	g_{N_0}(x,\tau,k)  \,
	d\tau
	\label{eq:IF2}
\end{align}
where 
\[
	g_0(x,\tau,k)
	= (\Psi^{r,\ell})^{(p)}(k^{\frac{1}{3}}Z(\tau)) \,
	F_{2}(x,\tau)
	\qquad
	\text{and}
	\qquad
	g_{s+1}(x,\tau,k)
	= D_\tau \left[ \dfrac{g_s(x,\tau,k)}{D_\tau \left[ \lambda_x(\tau) \right]} \right]
.
\]
As can be inductively seen, we have
\begin{equation} \label{eq:gN0}
	g_{N_0}(x,\tau,k) = \sum_{n=0}^{N_0} f_n(x,\tau) \, D_\tau^n \left[ g_0(x,\tau,k) \right]
\end{equation}
for some smooth functions $f_n$ such that $\operatorname{supp} f_n(x,\cdot) \subset \operatorname{supp} \rho_2(\cdot)$ for all $x \in O^{\rm IL}_m$. 
By Leibniz's rule
\[
	D_{\tau}^{n} g_{0}(x,\tau,k) 
	= \sum_{n_1=0}^{n}
	\binom{n}{n_1}
	D_{\tau}^{n_1} \left[ (\Psi^{r,\ell})^{(p)}(k^{\frac{1}{3}}Z(\tau)) \right]
	D_{\tau}^{n-n_1} \left[ F_2(x,\tau) \right]
\]
so that applying Faa di Bruno's formula yields 
\begin{multline} \label{eq:derg0}
	D_{\tau}^{n} g_{0}(x,\tau,k) 
	= (\Psi^{r,\ell})^{(p)}(k^{\frac{1}{3}}Z(\tau)) \,
	D_{\tau}^{n} \left[ F_2(x,\tau) \right]
	+ \sum_{n_1=1}^{n}
	\binom{n}{n_1}
	\\
	\left\{ \sum_{\mu \, \cdot \, \xi(n_1) \, = \, n_1}
	k^{\frac{|\mu|}{3}} \,
	(\Psi^{r,\ell})^{(p+|\mu|)}(k^{\frac{1}{3}}Z(\tau))
	\prod_{j = 1}^{n_1} \dfrac{j}{\mu_{j}!}
	\left( \dfrac{D_\tau^j \left[ Z(\tau)  \right]}{j!} \right)^{\mu_{j}}
	\right\}
	D_{\tau}^{n-n_1} \left[ F_2(x,\tau) \right].
\end{multline}
Since $|\mu| \le n_1$ for all $\mu = (\mu_1,\ldots,\mu_{n_1})$ with $\mu \cdot \xi(n_1) = n_1$, and since by Lemma~\ref{lemma:Psilr} 
\begin{equation*}
	\left| (\Psi^{r,\ell})^{(p)} (\tau) \right|
	\le C_{p} \left( 1 + \left| \tau \right| \right)^{-p},
	\qquad
	p \in \mathbb{Z}_+,
\end{equation*}
\eqref{eq:derg0} implies
\begin{equation*}
	\left| D_{\tau}^{n} g_{0}(x,\tau,k) \right|
	\le C_{S,F_2,n} \, (1+k)^{\frac{n}{3}}.
\end{equation*}
Therefore \eqref{eq:gN0} yields
\[
	\left| g_{N_0}(x,\tau,k) \right|
	\le C_{S,F_2,N_0} \, (1+k)^{\frac{N_0}{3}}.
\]
Accordingly \eqref{eq:IF2} entails
\[
	\left| I_{F_{2},p,r,\ell}(x,k) \right|
	\le C_{S,F_2,N_0} \, (1+k)^{-\frac{2N_0}{3}}
	\le C_{S,N,F_2,N_0} \, (1+k)^{-N}
\]
provided $3N \le 2N_0$. This proves \eqref{eq:simplifyfurther2} for $j=2$. 

\noindent
{\bf Part 1d:} Here we finally prove \eqref{eq:rapiddecaynew}. To this end, we observe that
\begin{equation*}
	D^\zeta_x D^n_k \, \left[ k \, I_j(x,k) \right]
	= k \ D^\zeta_x D^n_k \, I_j(x,k)
	+ \left\{
		\begin{array}{cl}
		0, & n = 0,
		\\
		n \ D^\zeta_x D^{n-1}_k \, I_j(x,k), & n \ge 1.
		\end{array}
	\right.
\end{equation*}
Therefore \eqref{eq:rapiddecaynew} is immediate from \eqref{eq:rapiddecaynew1} since \eqref{eq:umslow}
and $I_2,I_3 \in S^{-\infty}_{1,0}(O^{\rm IL}_{m,3\varepsilon} \times (0,\infty))$ imply that
\begin{equation} \label{eq:shorten}
	u^{\rm slow}_m - \frac{ik}{4} I_1
	= \frac{ik}{4} \left( I_2 + I_3 \right) 
	\in S^{-\infty}_{1,0}(O^{\rm IL}_{m,3\varepsilon} \times (0,\infty)).
\end{equation}

\noindent
{\bf Part 2}:
For the compact set $\mathcal{S}$ we initially fixed, here we show for all $k_{0} \in (0,\infty)$ and $N \in \mathbb{Z}_+$
\begin{equation}
\label{eq:bean}
	\left| u^{\rm slow}_{m}(x,k) - \sum_{p=0}^{N} k^{-p} \, A_{m,p}(x) \right|
	\le C_{\mathcal{S},N} \left( 1 + k \right)^{\mu_{N}},
	\quad
	\left( x, k \right) \in \mathcal{S} \times \left( k_{0}, \infty \right),
\end{equation}
where $\mu_{N} = - ( N + \frac{1}{2})$. Arguing as before, we may assume that $k_{0}$ is large enough so that,
for all $k > k_{0}$, the decomposition \eqref{eq:Hankelasymptotics} holds for all $(x,y)$ in the compact set $\mathcal{S} \times \partial K_{m}$.
Due to \eqref{eq:shorten}, it is sufficient to show that
\begin{equation} \label{eq:I1minusasymp}
	\left| \frac{ik}{4}I_1(x,k) - \sum_{p=0}^N k^{-p} \, A_{m,p}(x) \right|
	\le C_{\mathcal{S},N} \, (1+k)^{\mu_N},
	\quad
	\left( x, k \right) \in \mathcal{S} \times \left( k_{0}, \infty \right).
\end{equation}
To this end, we employ \eqref{eq:Hankelasymptotics} to define for $N \in \mathbb{Z}_+$ and $(x,k) \in O^{\rm IL}_{m} \times (0,\infty)$
\[
	I_{1,N}(x,k)
	= e^{-ik\psi_{m}(x)}
	\int_{\partial K_{m}} \!\!\! \rho_{1}(y) \, H_{1,N}(k|x-y|) \, \frac{x-y}{|x-y|} \cdot \nu(y) \, e^{ik \varphi_{m}(y)} \,
	\eta^{\rm slow}_{m}(y,k) \, ds(y).
\]
$\eta^{\rm slow}_m(y,k)$ is bounded independently of $y$ and $k$ because it lies in
$S^{0}_{\frac{2}{3},\frac{1}{3}}(\partial K_{m} \times (0,\infty))$.
Therefore \eqref{eq:Hankelasymptotics}-\eqref{eq:HankelTip}-\eqref{eq:HankelTail} imply for all $N \in \mathbb{Z}_+$
\begin{equation} \label{eq:bean1}
	\left| \frac{ik}{4} I_{1}(x,k) - \frac{ik}{4} I_{1,N}(x,k) \right|
	\le C_{N,\mathcal{S}}
	\left( 1 + k \right)^{\mu_N},
	\quad
	(x,k) \in \mathcal{S} \times (k_0, \infty).
\end{equation}
Further, in light of the illuminated region asymptotic expansion \eqref{eq:MT85a}, we define for $N \in \mathbb{Z}_+$ and
$(x,k) \in O^{\rm IL}_{m} \times (0,\infty)$
\begin{equation} \label{eq:JNdef}
	J_{N}(x,k)
	= e^{-ik\psi_{m}(x)}
	\sum_{j=0}^{N} k^{-j} 
	\int_{\partial K_{m}} \rho_{1}(y) \, H_{1,N}(k|x-y|) \, e^{ik \varphi_{m}(y)}
	a_{m,j}(y) \, ds(y).
\end{equation}
By construction $\operatorname{supp}(\rho_1)$ depends only on $\mathcal{S}$ so that Theorem~\ref{Thm:main1}(i) entails
\[
	\left| \eta^{\rm slow}_m(y,k) - \sum_{j=0}^N k^{-j} \, a_{m,j}(y) \right|
	\le C_{\mathcal{S},N} \, (1+k)^{-(N+1)},
	\qquad
	(y,k) \in \operatorname{supp}(\rho_1) \times (k_0,\infty).
\]
From \eqref{eq:HankelTip}, we also have
\[
	\left| H_{1,N}(k|x-y|) \right| \le C_{\mathcal{S},N} (1+k)^{-\frac{1}{2}},
	\qquad
	(x,y,k) \in \mathcal{S} \times \partial K_m \times (k_0,\infty).
\]
Therefore, for all $N \in \mathbb{Z}_+$,
\begin{equation} \label{eq:bean2}
	\left| \frac{ik}{4} \, I_{1,N}(x,k) - \frac{ik}{4} J_{N}(x,k) \right| \le C_{N,\mathcal{S}} 
	\left( 1 + k \right)^{\mu_N} ,
	\quad
	(x,k) \in \mathcal{S} \times (k_0, \infty). 
\end{equation}
Using \eqref{eq:HankelTip} in \eqref{eq:JNdef}, we obtain
\[
	\frac{ik}{4} J_{N}(x,k)
	= e^{-ik\psi_{m}(x)}
	\sum_{j=0}^N \sum_{s=0}^N
	k^{\frac{1}{2}-j-s} 
	g_{m,j,s}(x,k)
\]
where
\[
	g_{m,j,s}(x,k)
	= \frac{i}{4} \, c_{1,s}
	\int_{\partial K_{m}} e^{ik \, (\varphi_{m}(y)+|x-y|)} \left[ \rho_{1}(y) \, \frac{x-y}{|x-y|} \cdot \nu(y) \, \frac{a_{m,j}(y)}{|x-y|^{s+\frac{1}{2}}} \right] ds(y).
\]
Recall that, for any $x \in \mathcal{S}$, the only stationary point of the phase $\varphi_{m}(y)+|x-y|)$ in $\operatorname{supp}(\rho_1)$
is $y(x)$, $\varphi_{m}(y(x))+|x-y(x)|) = \psi(x)$, and $\rho_1(y(x)) = 1$ by construction. Using the definition \eqref{eq:fmjs2q} of $f_{m,j,s_2,q}$,
the stationary phase lemma therefore entails for $0 \le j,s \le N$ and $(x,k) \in \mathcal{S} \times (k_0,\infty)$
\begin{align*}
	\Big| g_{m,j,s} & (x,k) - e^{ik \, \psi(x)} \sum_{q=0}^N k^{-(q+\frac{1}{2})} f_{m,j,s,q}(x) \Big|
	\\
	& \le C_N \, (1+k)^{-(N+1)} \,
	\left\Vert \rho_1(y(t)) \, \frac{x-y(t)}{|x-y(t)|} \cdot \nu(y(t)) \, \frac{a_{m,j}(y(t))}{|x-y(t)|^{s_2+\frac{1}{2}}} \right\Vert_{C^{N+2}[0,P_m]}
	\\
	& \le C_{\mathcal{S},N} \, (1+k)^{-(N+1)}.
\end{align*}
This implies that if
\[
	\frac{ik}{4} \tilde{J}_{N}(x,k)
	= \sum_{j=0}^N \sum_{s=0}^N \sum_{q=0}^N
	k^{-j-s-q} 
	f_{m,j,s,q}(x),
\]
then
\begin{equation} \label{eq:bean3}
	\left| \frac{ik}{4} J_{N}(x,k) - \frac{ik}{4} \tilde{J}_{N}(x,k) \right|
	\le C_{\mathcal{S},N} \, (1+k)^{-(N+1)},
	\qquad
	(x,k) \in \mathcal{S} \times (k_0,\infty).
\end{equation}
Clearly, we also have
\begin{equation} \label{eq:bean4}
	\left| \frac{ik}{4} \tilde{J}_{N}(x,k) - \sum_{p=0}^N k^{-p} \, A_{m,p}(x) \right|
	\le C_{\mathcal{S},N} \, (1+k)^{-(N+1)},
	\qquad
	(x,k) \in \mathcal{S} \times (k_0,\infty).
\end{equation}
Therefore \eqref{eq:I1minusasymp} follows from \eqref{eq:bean1}, \eqref{eq:bean2}, \eqref{eq:bean3}, and \eqref{eq:bean4}.

\noindent
{\bf Part 3:} Since $A_{m,p}(x) \in C^{\infty}(O^{\rm IL}_m)$, we have $k^{-p} A_{m,p}(x) \in S^{-p}_{1,0}(O^{\rm IL}_m \times (0,\infty))$.
In light of Parts 1 and 2, the fundamental asymptotic expansion lemma therefore implies that
$u^{\rm slow}_{m}(x,k) \in S^{0}_{1,0} (O^{\rm IL}_{m} \times (0,\infty))$
and $u^{\rm slow}_{m}(x,k) \sim \sum_{p=0}^{\infty} k^{-p} \, A_{p}(x)$, and this completes the proof.

\section{Numerical validation}
\label{sect:5}

\begin{figure}[ptb]
\begin{center}
	\includegraphics*[width=3.0in,viewport=35 25 510 390,clip] {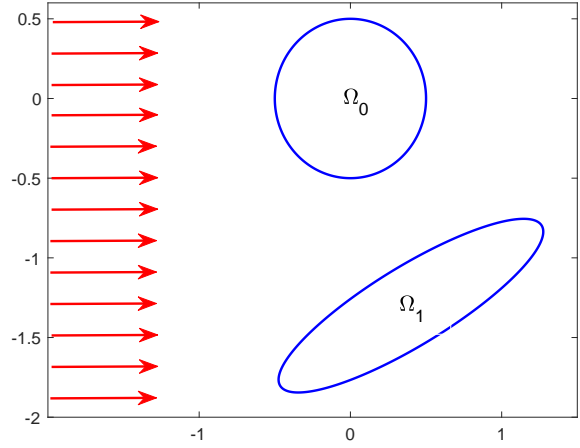}
\end{center}
\caption{A multiple scattering configuration comprised of a circle $\Omega_0$ and an ellipse
$\Omega_1$ illuminated by a plane wave incidence coming in from the left.
}
\label{Fig:Configuration}
\end{figure}

\begin{figure}[ptb]
\begin{center}
	\includegraphics*[width=5.1in,viewport=35 75 595 800,clip] {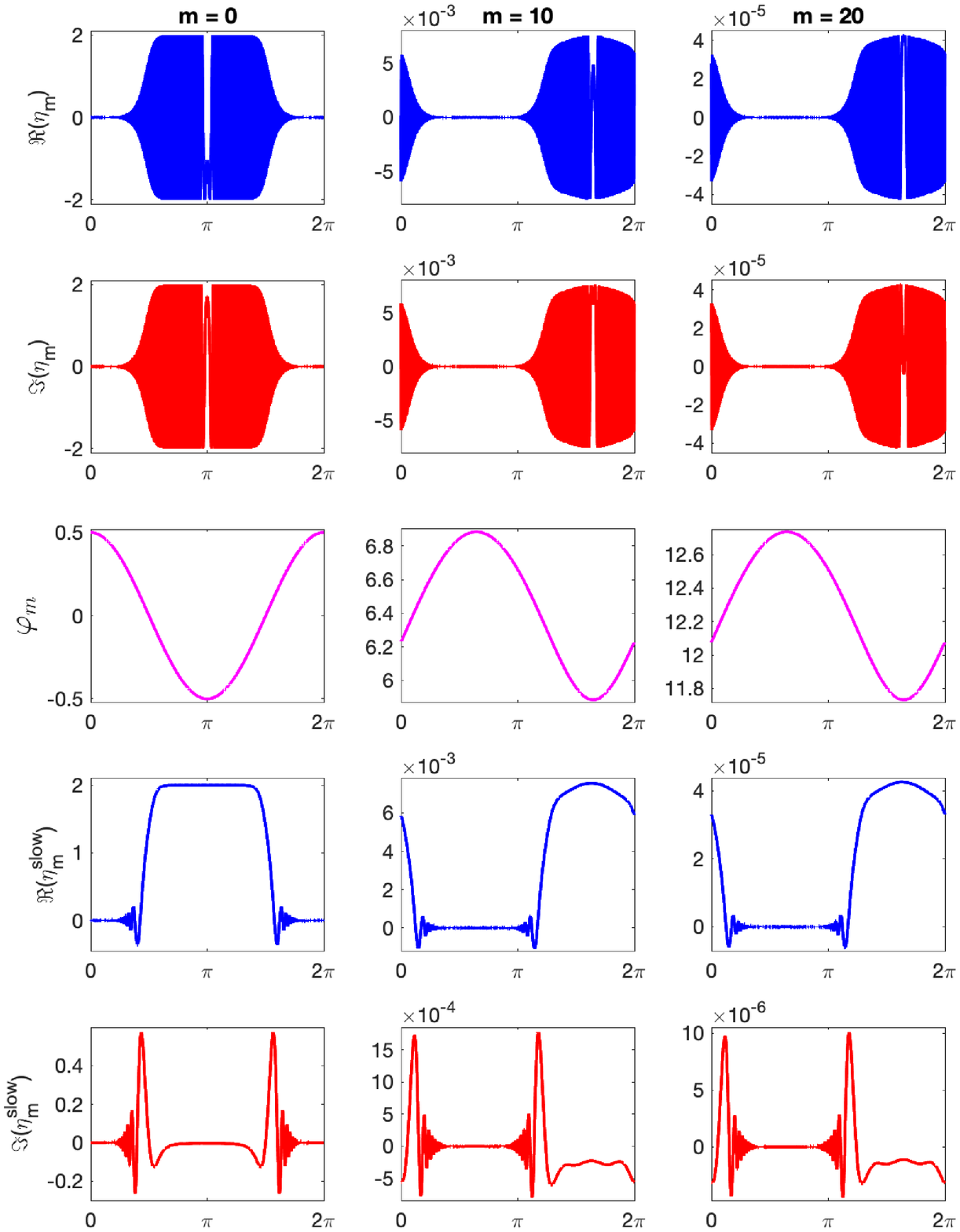}
\end{center}
\caption{The real $\Re(\eta_m)$ and imaginary $\Im(\eta_m)$ parts of $\eta_m$
(first and second rows), the phase function $\phi_m$ (third row), and the real $\Re(\eta_m^{\rm slow})$
and imaginary $\Im(\eta_m^{\rm slow})$ parts of $\eta_m^{\rm slow}$ (fourth and fifth rows)
for the wavenumber $k = 800$ and reflections $m=0$ (left pane), $m=10$ (middle pane),
and $m=20$ (right pane).
}
\label{Fig:Circles}
\end{figure}

\begin{figure}[ptb]
\begin{center}
	\includegraphics*[width=5.1in,viewport=25 75 595 803,clip] {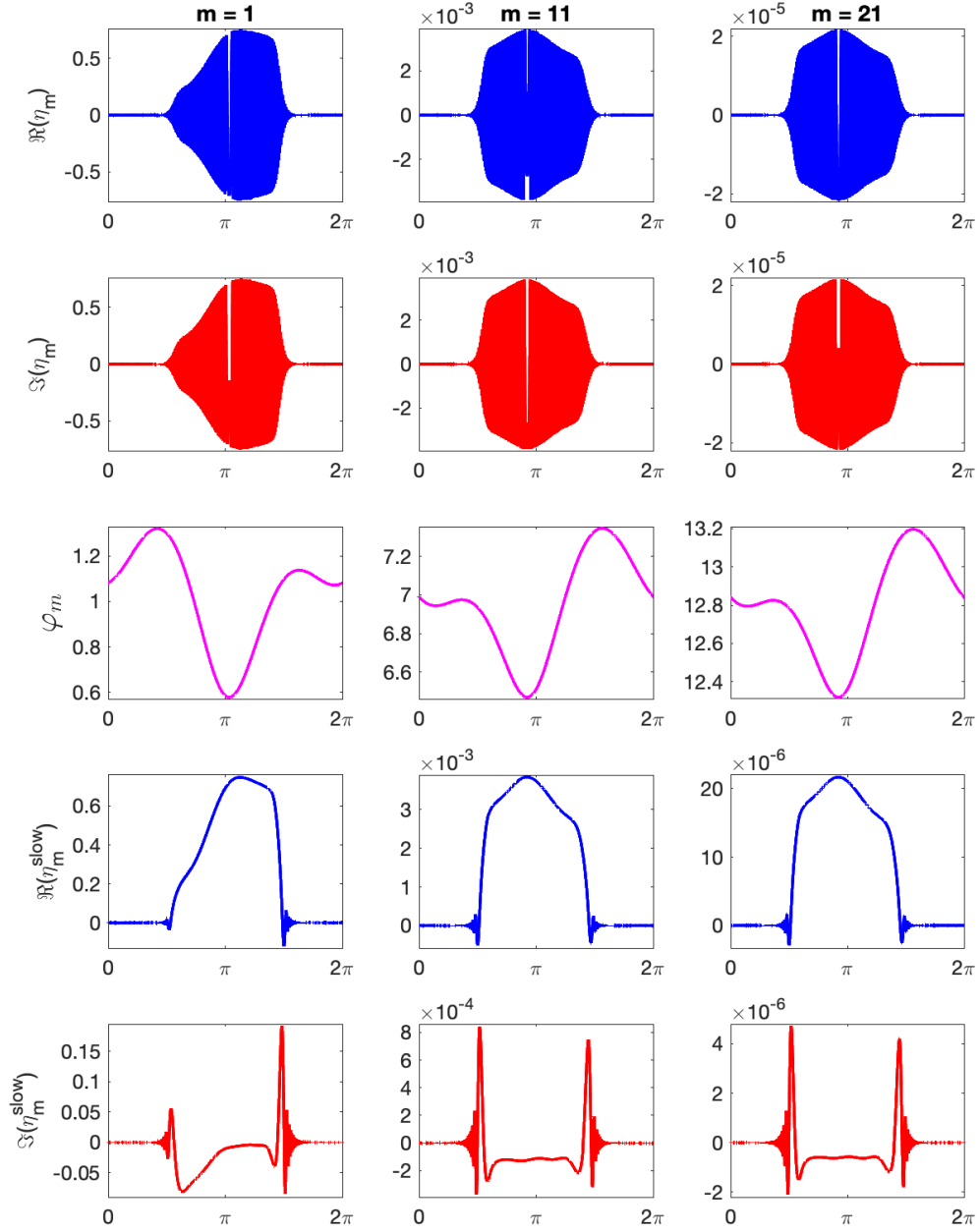}
\end{center}
\caption{The real $\Re(\eta_m)$ and imaginary $\Im(\eta_m)$ parts of $\eta_m$
(first and second rows), the phase function $\phi_m$ (third row), and the real $\Re(\eta_m^{\rm slow})$
and imaginary $\Im(\eta_m^{\rm slow})$ parts of $\eta_m^{\rm slow}$ (fourth and fifth rows)
for the wavenumber $k = 800$ and reflections $m=1$ (left pane), $m=11$ (middle pane),
and $m=21$ (right pane).
}
\label{Fig:Ellipses}
\end{figure}

To validate Theorem~\ref{Thm:main1} through numerical simulations, we consider a multiple
scattering configuration consisting of two smooth strictly convex scatterers
illuminated by a plane wave incidence coming in from the left ($\alpha = (1,0)$) as depicted in
Fig.~\ref{Fig:Configuration}. The first obstacle there, denoted as $\Omega_0$, is a circle of radius
$\frac{1}{2}$ centered at the origin and is taken with the parameterization
$x(t) = \frac{1}{2} \left( \cos t, \sin t \right)$ ($0 \le t < 2\pi$). The second scatterer, denoted as
$\Omega_1$, is an ellipse $x(t) = \left( \frac{1}{4} \cos t, \sin t \right)$ ($0 \le t < 2\pi$) rotated
by $60^\circ$ in the clockwise direction and translated by the vector $\frac{1}{10} \left( 4,-13 \right)$.
To illustrate Theorem~\ref{Thm:main1}, we consider the iterative solution of
integral equations \eqref{eq:inteqall} for the wavenumber $k = 800$ on the sequence of obstacles
$\{ K_m \}_{m \ge 0}$ where $K_{2m} = \Omega_0$ and $K_{2m+1} = \Omega_1$
so that $\eta_{m}$ is the total field generated on the circle $\Omega_0$
(respectively the ellipse $\Omega_1$) at the $m$-th iteration when $m$
is even (respectively odd).

Concerning the circle $\Omega_0$, in Fig.~\ref{Fig:Circles}, we display
the graphs of the real and imaginary parts of $\eta_m$, the phase function
$\phi_m$, and the real and imaginary parts of $\eta_m^{\rm slow}$
(cf. \eqref{eq:extractphases})
at iterations $m = 0$, $m = 10$, and $m = 20$. Similarly, in Fig.~\ref{Fig:Ellipses}
we display the same data corresponding to the ellipse $\Omega_1$ at
iterations $m = 1$, $m = 11$, and $m = 21$. 

The simulations depicted in Figures~\ref{Fig:Circles} and ~\ref{Fig:Ellipses} are
in agreement with the asymptotic expansions of the envelops $\eta_m^{\rm slow}$
presented in Theorem~\ref{Thm:main1} as they support that $\eta_m^{\rm slow}$
admits a classical asymptotic expansion (cf. Theorem~\ref{Thm:main1}(i)) in the illuminated region
$\partial K_m^{\rm IL}$ \eqref{eq:illum_m_def}. This behavior transforms through a change
in its asymptotic behavior (cf. Theorem~\ref{Thm:main1}(ii)) across the shadow boundaries
$\partial K_m^{\rm SB}$ \eqref{eq:shbound_m_def} to rapid decrease
(cf. Theorem~\ref{Thm:main1}(iii)---but with additional oscillations not captured by the
phase extraction) as one moves deep into the shadow region $\partial K_m^{\rm SR}$
\eqref{eq:shreg_m_def}.

\section{Conclusions}
\label{sec:conclude}
In this paper, we derived the asymptotic expansions of the 
solutions of multiple scattering problem with the Neumann boundary condition.
These expansions allowed for the derivation of sharp
wavenumber dependent estimates related to their derivatives,
more generally on the derivatives of envelopes obtained by
subtracting finitely many terms in their asymptotic expansions
as presented in Theorem~\ref{thm:etamder}.
These estimates, in turn, can be used to extend the Galerkin
boundary element methods for sound hard single scattering
problems to multiple scattering scenarios
for frequency independent implementations.


\bibliographystyle{plain}      
\bibliography{BoubendirEcevit2022.bib}   
\end{document}